\documentclass[12pt]{amsart}
\usepackage[margin=1in]{geometry}
\usepackage{amsmath,amsfonts,amsthm,amssymb,bbm}
\usepackage{graphicx,color,dsfont}
\usepackage{enumitem}
\usepackage{fourier}

\usepackage{amsfonts}
\usepackage{amsthm}
\usepackage{amsmath}
\usepackage{amscd}
\usepackage{amssymb}
\usepackage{graphics}
\usepackage{graphicx}
\usepackage{tikz}
\usetikzlibrary{arrows,shapes}
\usetikzlibrary{arrows.meta}

\newtheorem{theorem}{Theorem}
\newtheorem{lemma}{Lemma}
\newtheorem{proposition}{Proposition}

\newtheorem{definition}{Definition}
	
	\newtheorem{corollary}{Corollary}

\newtheorem{claim}{Claim}

 \newtheorem{thm}{Theorem}[section]
 \newtheorem{cor}[thm]{Corollary}

 \theoremstyle{definition}
 
 \theoremstyle{remark}

 \numberwithin{equation}{section}

\newcommand{\vertiii}[1]{{\left\vert\kern-0.25ex\left\vert\kern-0.25ex\left\vert #1
    \right\vert\kern-0.25ex\right\vert\kern-0.25ex\right\vert}}

\newcommand{\f}[2]{\frac{#1}{#2}}

\newcommand{\cl}{{\mathcal L}}

\newcommand{\n}[2]{{\left\| #1 \right\|}_{#2}}

\newcommand{\al}{\alpha}
\newcommand{\be}{\beta}

\newcommand{\ve}{\varepsilon}

\newcommand{\ka}{\kappa}
\newcommand{\la}{\lambda}

\newcommand{\si}{\sigma}

\newcommand{\vp}{\varphi}

\newcommand{\rone}{\mathbb R}

\newcommand{\lan}[1]{\left\langle #1\right\rangle}

\newcommand{\dpr}[2]{\langle #1,#2 \rangle}

\newcommand{\ct}{\mathbb T}

\newcommand{\cs}{\mathcal S}

\newcommand{\cf}{\mathcal F}

\newcommand{\cd}{\mathcal D}

\newcommand{\ch}{\mathcal H}

\newcommand{\p}{\partial}

\newcommand{\wt}{\widetilde}

\newcommand{\beq}{\begin{equation}}
\newcommand{\eeq}{\end{equation}}
\newcommand{\beqna}{\begin{eqnarray*}}
\newcommand{\eeqna}{\end{eqnarray*}}
\newcommand{\beqn}{\begin{equation*}}
\newcommand{\eeqn}{\end{equation*}}
\newcommand{\bp}{\begin{proof}}
\newcommand{\ep}{\end{proof}}
\newcommand{\bprop}{\begin{proposition}}
\newcommand{\eprop}{\end{proposition}}
\newcommand{\bt}{{\mathbf T}}
\newcommand{\et}{\end{theorem}}
\newcommand{\bex}{\begin{Example}}
\newcommand{\eex}{\end{Example}}
\newcommand{\bc}{\begin{corollary}}
\newcommand{\ec}{\end{corollary}}
\newcommand{\bcl}{\begin{claim}}
\newcommand{\ecl}{\end{claim}}
\newcommand{\bl}{\begin{lemma}}
\newcommand{\el}{\end{lemma}}

\newcommand{\cj}{{\mathcal J}}

\begin{document}

\title
[Dynamics of the Drinfeld-Sokolov-Wilson system]{Dynamics of the Drinfeld-Sokolov-Wilson system: well-posedness and (in)stability of the traveling waves}

 \author[Ognyan Christov]{\sc Ognyan Christov}

\address{Ognyan Christov, Department of Mathematics,Faculty of Mathematics and Informatics, Sofia University, 1164 Sofia, Bulgaria}
\email{christov@fmi.uni-sofia.bg}

 	\author{Sevdzhan Hakkaev} 
 \address{Sevdzhan Hakkaev, ORCID 0000-0002-1531-8340, Department of Mathematics, Faculty of Science, Trakya University, \\ 22030 Edirne, Turkey and \\ Institute of Mathematics and Informatics, \\ Bulgarian Academy of Sciences, Acad. G. Bonchev Str. bl. 8, 1113 Sofia, Bulgaria}
 \email{s.hakkaev@math.bas.bg}

 \author{SEUNGLY OH}
 
 \address{Seungly Oh, Department of Mathematics, Western New England University, 1215 Wilbraham Road\\
 	Springfield, Massachusetts, 01119, USA}
 \email{seungly.oh@wne.edu}

 \author[Atanas G. Stefanov]{\sc Atanas G. Stefanov}
 \address{Atanas G. Stefanov, ORCID 0000-0002-0934-506X\\  Department of Mathematics, University of Alabama - Birmingham,
 	1402 10th Avenue South
 	Birmingham AL 35294, USA.
 }
\email{stefanov@uab.edu}


\date{\today}

\thanks{Stefanov is partially supported by NSF-DMS under grant \# 220478.}

\begin{abstract}
   We analyze the Drinfeld-Sokolob-Wilson system, which features a dispersive, KdV type evolution with a dispersionless conservation law.  We establish well-posedness with low regularity initial data $L^2({\mathbb T})\times L^2({\mathbb T})$ for the Cauchy problem on periodic background, which is then extrapolated to global solutions, due to $L^2$ conservation law. We also establish a dynamically more relevant result, namely  a global persistence of solutions with (large) initial data in  $H^1({\mathbb T})\times L^2({\mathbb T})$.  This is obtained by following a more sophisticated approach, specifically the method of normal forms. 
   
   Finally, for a fixed period $L$, we construct an explicit one parameter  family of periodic waves, see \eqref{2.16} below.  We  show that they are all spectrally unstable with respect to co-periodic perturbations. Specifically, we show that the Hamiltonian instability index is equal to one, which identifies the instability as a single positive growing mode.

\end{abstract}

\maketitle

  \section{Introduction}
  We consider  the following system 
  \begin{equation}\label{eq:1}
  	\left\{ \begin{array}{l} u_t + (uv)_x  + u_{xxx} = 0,\\
  		v_t + uu_x  = 0\\ 
  		u|_{t=0} = u_0 \in L^2_{per.}(0,L), \quad v|_{t=0} = v_0 \in L^2_{per.}(0,L). \end{array}\right.   
  \end{equation}
   The system was derived as a model of   strong interactions of two-dimensional long internal gravity waves, \cite{BCS1, BCS2}, see also a more ad-hoc derivation in the earlier works \cite{DS, Wil}. It is worth noting that while the $u$ equation displays standard dispersive properties, it is coupled with the $v$ equation, which is a simple mass conservation, and by itself, does not exhibit dispersive behavior. Our goal in this project is two fold - we would like to study the well-posedness of this system   as well as the existence and spectral stability    of traveling wave (TW) solutions. 
   
   \subsection{Well-posedness of the DSW system}
   We discuss some specifics related to the well-posedness properties of the system \eqref{eq:1}. This will be done, as said above,  in the standard scale of Sobolev spaces. We however list the relevant conservation laws, which are usually indicative of the dynamical picture. More concretely, 
   the conserved quantities\footnote{Henceforth, the notation $\int f dx$ will be a short-cut for an integration over one space period, $\int_0^L f(x) dx$, unless otherwise indicated} for classical solutions are 
   \begin{equation}
   	\label{7} 
   	  \int_0^L u dx;  \int_0^L  v dx; 
   	   \int_0^L   (u_x^2 – u^2 v)dx; 
   	  \int_0^L   (u^2+ v^2) dx
   \end{equation}
   We are now ready to state our main well-posedness results, modulo some minor technical definitions, which will be provided at a later point. 
   \begin{theorem}[Local well-posedness in $L^2$]
   	\label{th:wp}
 Let $0\leq s \leq \f{1}{2}$. 	Given $u_0 \in H^s_{per.}(0,L)$ and $v_0\in L^2_{per.}(0,L)$, we can find $T>0$ depending only on $\n{u_0}{H^s}+\n{v_0}{L^2}$ so that there exist a solution for \eqref{eq:1}, $(u, v) \in C^0_t H^s_x \times C^0_t L^2_x$ on where $t\in [0,T]$.  This solution is unique within an auxiliary space $(u,v) \in Y\hookrightarrow C^0_t H^s_x \times C^0_t L^2_x$ to be defined below.  Furthermore, the solution map $(u_0, v_0)\mapsto (u,v)$ is  continuous in $H^s\times L^2 \to C^0_t H^s_x \times C^0_t L^2_x$ topology.  
     \end{theorem} 
 The local well-posedness stated above assumes that the regularity of $v_0$ is fixed at $L^2$, while $u_0$ can take a higher regularity.  The standard bilinear estimate method used to obtain this result limits the regularity of $u$ to be at most $H^{\f{1}{2}}$. The method of our proof can be adapted to produce an analogous and slightly more general well-posedness statement for initial data $(u_0, v_0) \in H^{s_1} \times H^{s_2}$ where $0\leq s_2 \leq s_1 \leq s_2 + \f{1}{2}$.\\

A particular case of interest for Theorem~\ref{th:wp} is when $s=0$, where $(u_0, v_0) \in L^2 \times L^2$.  Due to the conserved quantity $\int u^2+v^2$, the local wellposedness statement in this case leads directly to a corresponding global-in-time wellposedness statement.

\begin{corollary}\label{cor:wp}
Given initial data $(u_0, v_0)\in L^2_x\times L^2_x$, the solution given in Theorem~\ref{th:wp} extends to a global-in-time solution, so that $(u(t), v(t)) \in C^0_t L^2_x \times C^0_t L^2_x$ for all $t \geq 0$. 
\end{corollary}

Next, we consider the solution with initial data $(u_0, v_0) \in H^1 \times L^2$ which is not covered by the statement of Theorem~\ref{th:wp}.  While we can obtain a unique solution $(u(t), v(t)) \in C^0_t L^2_x \times C^0_t L^2_x$, we want to claim $(u(t), v(t)) \in C^0_t H^1_x \times C^0_t L^2_x$ for $t\geq 0$.  We cannot achieve this result solely based on standard  bilinear estimates.  We will employ a normal form transformation approach to achieve this smoothing effect \cite{shatah}.  Typically, normal form transformation, or equivalently differentiation-by-parts \cite{titi, tzirakis} is used to transform a quadratic nonlinearity to a cubic one \cite{shatah}.  However, another useful application is to use the normal form to filter out the roughest portion of a nonlinearity \cite{stefanov}, thereby enabling an estimate in a much smoother norm for the solution.  We take this approach to place $u(t)$ in the target space $C^0_t H^1_x$.
 
   \begin{theorem}(Persistence of regularity in $H^1\times L^2$)
   	\label{prop:109} 
   	
   	Let  $(u_0, v_0)\in H^1(0,L)\times L^2(0,L)$. Then, the global-in-time solution $(u(t),v(t))$ given in Corollary~\ref{cor:wp} also lives in $C^0_t H^1_x \times C^0_t L^2_x$ for all $t\geq 0$.  More specifically, the solution $(u,v)$ satisfies
$$
\sup_{t\geq 0} \n{u(t)}{H^1_x} + \n{v(t)}{L^2_x} \lesssim C\left(\n{u_0}{H^1}, \n{v_0}{L^2}\right)
$$
   \end{theorem}
 
The local well-posedness, as indicated above,  is achieved through the technical method of normal forms, see Section \ref{sec:4}.   Let us briefly comment on the specifics of the approach. 
 
 Starting with $(u_0, v_0)\in H^1(\ct)\times L^2(\ct)$, we reduce, in a standard manner,    to zero mean data for $u$.  Next, one is tempted to  produce local solutions for \eqref{eq:1} via a fixed point method in  the KdV based Bourgain space $u\in X^{1, \f{1}{2}}(\ct)$. This is however easily seen to fail (in fact the best one can say is that $u\in X^{\f{1}{2}, \f{1}{2}}$), see the discussion right before Lemma \ref{le:7}. The fix is to seek a ``normal form'' decomposition, which resolves the worst behaved terms (in the $X^{s,b}$ scale) via smooth solutions in addition to a  remainder term, which is  indeed placed in  $X^{1, \f{1}{2}}(\ct)$, see \eqref{eq:3}.

   \subsection{Periodic solutions}
 We now discuss the existence of periodic TW solutions. As it turns out, there exists a rich family of those, for different values of the parameters, each with different properties. 
   Assume $c > 0$. As usual we put
   \begin{equation}
   	\label{2.2}
   	\xi = x - c t, \quad u (t, x) = \psi (\xi), \quad v (t, x)
   	= \phi (\xi) .
   \end{equation}
   Then substituting these expressions into the system \eqref{eq:1},  we obtain
  \begin{equation}
  		\label{2.3} 
  		  	\left\{ \begin{array}{l}
  	- c \psi'  + \psi' \phi + \psi \phi'  + \psi''' = 0, \\
  	- c \phi'  + \psi \psi' = 0.
  \end{array}\right.
  \end{equation}
   From the second equation of (\ref{2.3}) we get
   $$
   -c \phi + \frac{\psi^2}{2} = D_1
   $$
   Taking  $D_1=0$ yields\footnote{Here, taking $D_1\neq 0$ is an option and it will not create any extra difficulties, as it can be absorbed in the speed parameter $c$. In order to simplify the exposition, we take $D_1=0$} 
   \begin{equation}
   	\label{2.4}
   	\phi =  \frac{\psi^2}{2 c} .
   \end{equation}
 Substituting this expression into the first equation of (\ref{2.3}) yields an exact derivative equal to zero. Hence, upon introdusing an extra constant $F_1$, we may integrate this to the form 
 \begin{equation}
 	\label{2.5}
 	\psi'' + \frac{\psi^3}{2 c} - c \psi  = F_1
 \end{equation}
 Multiplying by $\psi'$ and realizing once again that this is an exact derivative, we can further integrate once more. Upon introducing an extra free constant $E$,   we get
 \begin{equation}
 	\label{2.6}
 	\frac{1}{2} (\psi')^2 + \frac{1}{8 c} \psi^4 - \frac{c}{2} \psi^2 - F_1 \psi = E .
 \end{equation}
To simplify matters, we take $E=0$. The last equation takes the  form
\begin{equation}
	\label{2.7}
	\frac{1}{2} (\psi')^2  + U (\psi) = E, 0\leq x\leq L, \quad \quad U (\psi) = \frac{\psi}{8 c} (\psi^3 - 4 c^2 \psi - 8 c F_1 )
\end{equation}
 This is a well-known Newton type equation for the unknown function $\psi$, which takes the form of a quartic polynomial. There are several possible family of solutions to \eqref{2.7}, depending on the zeroes of the quartic $U$. The further, somewhat subtle analysis of these solutions will be performed in the subsequent sections. We state a preliminary result, based on our progress so far. 
\begin{proposition}
	\label{prop:10} 
	Let $c>0$. Suppose that $\psi$ is a smooth, $L$-periodic solution to \eqref{2.7}. Then, the pair 
	\begin{equation}
		\label{16} 
			(u,v)=(\psi(x- ct), \phi(x-c t))=(\psi(x- ct), \f{\psi^2(x-c t)}{2c})
	\end{equation}
	is a solution to \eqref{eq:1}. In addition, this is a one-parameter family of such functions. More precisely, we can express the 
	one-parameter family of $L$ periodic functions 
	 \begin{equation}
		\label{2.16}
		\psi (\xi) = \eta_4 \frac{ dn^2 (\frac{2K(\ka) \xi}{L}, \kappa)}{1 + \beta^2 sn^2  (\frac{2K(\ka) \xi}{L}, \kappa)}, \quad \psi (0) = \eta_4 ,
	\end{equation}
	where 
	we have  introduced the convenient parameters (with $\ka\in (0,1)$), 
	\begin{eqnarray}
		\nonumber h(k) &=&  4 \sqrt{1-k^2+k^4} \\
			\label{2.27}
		c  &=&  \frac{4 {\rm K}^2 (\kappa)}{L^2} h (\kappa); \\ 
			\label{2.28}
		\eta_4 &=&  \frac{8 \sqrt{2} {\rm K}^2 (\kappa)}{\sqrt{3} L^2} \sqrt{h (\kappa)[h (\kappa) + 2 (2\kappa^2 -1)]} ;\\
			\label{2.31}
		\beta^2 &=&  2 \kappa^2 \frac{\sqrt{h (\kappa) + 2(2 \kappa^2 -1)}}{\sqrt{h (\kappa) + 2(2 \kappa^2 -1)} + \sqrt{3} \sqrt{h (\kappa) - 2(2 \kappa^2 -1)}}\\
		\label{2.a10} 
		F_1 &=& -\f{\eta_4}{8 c}(4c^2-\eta_4^2).
	\end{eqnarray}
	
\end{proposition}
Our next focus will be on the stability of these solutions. 
\subsection{On the stability of the periodic TW solutions}
For a pair $\phi, \psi$, solutions to \eqref{2.7}, as described in Proposition \ref{prop:10}, we first perform the linearization. Specifically,  with the ansatz 
$$
u(t,x)=\psi(x- c t)+e^{\la t} U(x-c t), v(t,x)=\phi(x- c t)+e^{\la t} V(x-c t)
$$
which we place in \eqref{eq:1}, we expand and ignore $O(U^2+V^2)$. As a result, we obtain 
\begin{equation}
	\label{40} 
 \p_\xi \left(\begin{array}{cc}
	 -\p_{\xi\xi}+c - \phi & -\psi \\
	 -\psi & c
\end{array}\right)  \left(\begin{array}{c}
U \\ V 
\end{array}\right) =	\la \left(\begin{array}{c}
U \\ V 
	\end{array}\right)
\end{equation}
We introduce some more notations, namely 
\begin{eqnarray*}
	\cl_- &:= &  -\p_{\xi\xi}+c - \phi  =  -\p_{\xi\xi}+c - \f{\psi^2}{2c}, \\
	\cl_+ &:=& -\p_{\xi\xi}+c -  \f{3 \psi^2}{2c}\\ 
	\ch &:= &  \left(\begin{array}{cc}
		-\p_{\xi\xi}+c - \phi & -\psi \\
		-\psi & c
	\end{array}\right) =  \left(\begin{array}{cc}
		\cl_-& -\psi \\
		-\psi & c
	\end{array}\right) 
\end{eqnarray*}
 That is,   for the solution \eqref{16}, the linearized problem is in the form 
 \begin{equation}
 	\label{41} 
 	\p_\xi \ch \vec{Z}=\la \vec{Z}, 
 \end{equation}
  As is standard, we have the following definition. 
     \begin{definition} 
     	\label{defi:10} 
     	The periodic wave $\Psi=(\psi,\phi)$ is said to be spectrally stable (with respect to co-periodic perturbations), if the eigenvalue problem \eqref{41}, with $D(\ch)=H^2(0,L)\times L^2(0,L)$, does not have any solutions $(\la, \vec{Z}): \Re\la>0$. Otherwise, the $\Psi$ is spectrally unstable.

   \end{definition}
   It is then time to state our main result. 
   \begin{theorem}
   	\label{theo:15} 
   	Let $L>0$, and $(\psi, \vp)$ is the one-parameter family of $L$ periodic traveling wave solutions,  described in Proposition \ref{prop:10}.  
   Then, all of these waves are spectrally unstable in the sense of Definition \ref{defi:10}. Specifically, the instability is due to exactly one real unstable mode. That is,  there is a real $\la>0$ and non-trivial $\vec{Z}\in Dom(\p_\xi \ch)$, so that \eqref{41} holds true. 
   \end{theorem}
 The plan for our work is as follows. In Section \ref{sec:2}, we present a detailed description  of the explicit solutions, in terms of various parameters. The results therein are well summarized in Proposition \ref{prop:10}. Next, in Section \ref{sec:3}, we develop the local well-posedness theory for the DSW $L^2\times L^2$, the results of which are in Theorem \ref{th:wp}. In this section, we present many relevant estimates in periodic KdV Bourgain spaces, so it serves as a preliminary step for the more complicated situation for persistence of the solution in $H^1\times L^2$. This is done in full detail in Section \ref{sec:4}, wherein in addition to the technical results from Section \ref{sec:3}, we introduce  the normal form decomposition and prove the relevant estimates. In Section \ref{sec:5}, we give some preparatory results, needed for the stability of the waves. Specifically, we establish  that the Morse index of $\cl_+, \ch$ is two, and also, we prove that the kernel of both is one dimensional. This is followed  by the calculation of the generalized kernel of $\p_x \ch$, which allows us to compute  the elements of a matrix $\cd$, which comes up in the definition of the Hamiltonian instability index. A short index for the Hamiltonian index, follows by the construction of the inverse $\cl_+^{-1}$  is done in Section \ref{sec:6}. This allows to finally compute the elements of the matrix $\cd$ in terms of explicit integrals of elliptic Jacobi functions. This is done using Mathematica, as the formulas tend to get excessively long. In a nutshell, the graphs of $det(\cd)$ shows that $\cd$ has exactly one negative eigenvalue for all values of the parameters, which allows one to infer spectral instability of the waves.

 \section{Construction of the solutions}
 \label{sec:2} 
 In this section, we provide a detailed derivation of the solutions introduced in Proposition \ref{prop:10} inlcuding the relevant parametrizations, which be helpful in the sequel. To this end, we start with the equation \eqref{2.7}. Suppose now that $27 F_1 ^2 - 4 c^4 < 0$. Clearly, in this case we
 have four real roots for the equation $U (\psi) = 0$.
 
 We assume that $F_1 < 0$ in order that the periodic solution does not pass through zero.
 Then, an element of the phase portrait of (\ref{2.7}) is depicted on Fig. 1,
 which justify existence of positive periodic solutions $\eta_3 \leq \psi (\xi) \leq \eta_4$. Then, an element of the phase portrait of (\ref{2.7}) is depicted on Fig. 1,
 which justify existence of positive periodic solutions $\eta_3 \leq \psi (\xi) \leq \eta_4$.

 \begin{figure}[ht]
 	\begin{tikzpicture}
 		\draw [help lines, ->] (-6.5,0)--(6.5,0);
 		\draw [help lines, ->] (-6.5,-4)--(6.5,-4);
 		\draw [help lines, ->] (0,-6.1)--(0,3.0);
 		
 		\draw (-4,1) .. controls (-3,0) and (-2,-2) .. (0,0)
 		.. controls (1,1) and (2,2) .. (3,0)
 		.. controls (3.5,-1) and (4.5,-3) .. (6,2);
 		
 		\node [below right] at (6.5,0) {$u$}; \node [below right] at
 		(6.5,-4) {$u$}; \node [left] at (0,3.0) {$U (\psi)$}; \node [left]
 		at (0,-1) {$u'$};
 		
 		\node [below left] at (3,0) {$\eta_3$};
 		\node [below left] at (3,-4) {$\eta_3$};
 		
 		\node [below right] at (5.3,-4) {$\eta_4$};
 		\node [below right] at (5.3,0) {$\eta_4$};
 		
 		\node [below right] at (-3.5,0) {$\eta_1$};
 		
 		\node [above right] at (-6.3,-4) {$F_1 < 0$};
 		\node [above right, red] at (3.3,0) {$E=0$};
 		
 		\draw [dashed] (3,0)--(3,-4);
 		\draw [dashed] (5.3,0)--(5.3,-4);
 		\draw [red, line width=1.2] (3,0) -- (5.3,0);
 		
 		\draw [red, line width=1.2] (4.15,-4) ellipse (1.15 and 0.8);
 		
 		\draw[line width=0.2mm, -{Stealth[red]}] (4.15,-3.2) -- (4.15001,-3.2) node[midway, above=0.1pt] {};
 		
 		\draw[line width=0.2mm, -{Stealth[reversed,red]}] (4.15,-4.8) -- (4.15001,-4.8) node[midway, above=0.1pt] {};
 		
 	\end{tikzpicture}
 	\caption{The positive periodic solution of (\ref{2.7}).}
 \end{figure}
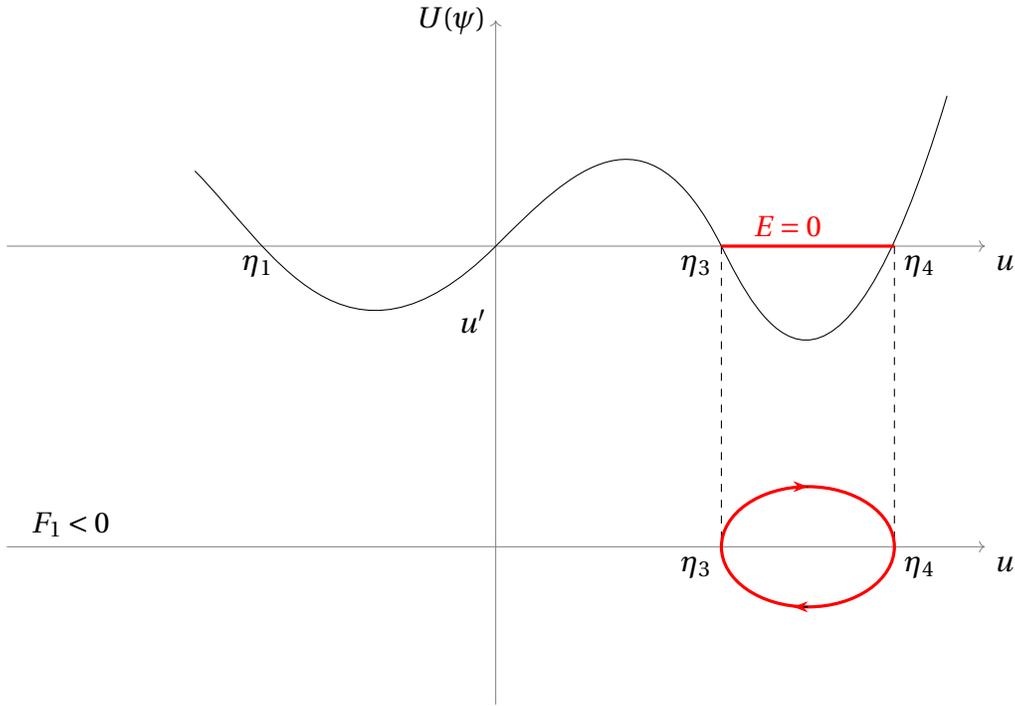
 Next, we are looking for explicit formulas for periodic solutions
 $\psi (\xi)$ in the period domain $[0, L]$ subject to the
 conditions $\psi (0) = \eta_4$ and $\psi (\sigma) = \eta_3$ for
 some $\sigma \in (0, L)$.
 We rewrite (\ref{2.7}) as follows 
 \begin{equation}
 	\label{2.8}
 	\frac{1}{2} (\psi')^2  =  - U (\psi) = \frac{1}{8 c}(\eta_4 -
 	\psi)(\psi - \eta_3) \psi (\psi - \eta_1).
 \end{equation}
 Clearly,
 \begin{align}
 	\label{2.9}
 	& \eta_1  + \eta_3 + \eta_4 = 0, \nonumber \\
 	& \eta_1 \eta_3  + \eta_1 \eta_4 + \eta_3 \eta_4 = - 4 c^2,  \\
 	&  \eta_1 \eta_3 \eta_4 = 8 c F_1, \nonumber
 \end{align}
 From (\ref{2.8}) one gets
 \begin{equation}
 	\label{2.10}
 	\frac{d \psi}{\sqrt{(\eta_4 - \psi)(\psi - \eta_3)\psi (\psi - \eta_1)}} = \frac{1}{2 \sqrt{c}} d \xi.
 \end{equation}
 On the other hand, the substitution
 \begin{equation}
 	\label{2.11}
 	sn^2 (y, \kappa) = \frac{(\eta_3 - \eta_1)(\eta_4 - \psi)}{(\eta_4 -\eta_3)(\psi - \eta_1)}
 \end{equation}
 transforms the left hand side of (\ref{2.10}) into (Bird \& Friedman \cite{BirdFriedman}, formula 257.00)
 \begin{equation}
 	\label{2.12} \frac{d \psi}{\sqrt{(\eta_4 - \psi)(\psi - \eta_3) \psi (\psi - \eta_1)}} = a d y,
 \end{equation}
 where
 \begin{equation}
 	\label{2.13}
 	a = \frac{2}{\sqrt{\eta_4 (\eta_3 - \eta_1)}}, \quad
 	\kappa^2 =  - \frac{\eta_1}{\eta_4} \frac{(\eta_4 - \eta_3)}{(\eta_3 - \eta_1)} .
 \end{equation}
 Both relations (\ref{2.10}) and (\ref{2.12}) give
 $
 y = \frac{\xi}{2 a \sqrt{c}},
 $
 and hence the final formula, 
 \begin{equation}
 	\label{2.14}
 	\psi (\xi) = \frac{\eta_4 (\eta_3 - \eta_1) + \eta_1 (\eta_4 - \eta_3) sn^2 (\frac{\xi}{2 a \sqrt{c}}, \kappa)}
 	{(\eta_3 - \eta_1) +  (\eta_4 - \eta_3) sn^2 (\frac{\xi}{2 a \sqrt{c}}, \kappa)}.
 \end{equation}
 Introduce the auxiliary parameter 
 \begin{equation}
 	\label{2.15}
 	\beta^2 : = - \kappa^2 \frac{\eta_4}{\eta_1} > 0.
 \end{equation}
 Then,  \eqref{2.14} can be rewritten in the more compact form 
 \begin{equation}
 	\label{2.16}
 	\psi (\xi) = \eta_4 \frac{ dn^2 (\frac{\xi}{2 a \sqrt{c}}, \kappa)}{1 + \beta^2 sn^2 (\frac{\xi}{2 a \sqrt{c}}, \kappa)}, \quad \psi (0) = \eta_4 ,
 \end{equation}
 where $sn, cn,  dn$ are the Jacobi  elliptic functions. 
 Since $\psi (\xi)$ is required to have a minimal period we conclude that
 $ L = 2 a \sqrt{c} (2 {\rm K} (\kappa))$,
 where ${\rm K} (\kappa)$ is the complete elliptic integral of first kind.
  
 We are interested mainly in $\psi (\xi)$ and its properties.
 Combining the first two identities from (\ref{2.9}) yields
 \begin{equation}
 	\label{2.18}
 	\eta_3 ^2 + \eta_4 ^2 + \eta_3 \eta_4 = 4 c^2, 
 \end{equation}
 from where we infer that
 \begin{equation}
 	\label{2.19}
 	0 < \eta_3 < \frac{2 c}{\sqrt{3}} < \eta_4 < 2 c.
 \end{equation}
 Furthermore,  we can obtain also from (\ref{2.9}) the expressions of $\eta_1$ and $\eta_3$ as functions of $c$ and $\eta_4$ as follows
 \begin{align}
 	\label{2.20}
 	& \eta_1  = -\frac{1}{2} \left(\sqrt{16 c^2 - 3 \eta_4 ^2} + \eta_4 \right), \nonumber \\
 	& \eta_3   = \frac{1}{2} \left(\sqrt{16 c^2 - 3 \eta_4 ^2} - \eta_4 \right),  \\
 	& F_1 = - \frac{\eta_4}{8c} (4c^2 - \eta_4 ^2) \nonumber
 \end{align}
 Using these expressions we find that
 \begin{equation}
 	\label{2.21}
 	\kappa^2 = \frac{\sqrt{16c^2 \eta_4 ^2 - 3 \eta_4 ^4} + 3 \eta_4 ^2 - 8c^2}{2 \sqrt{16c^2 \eta_4 ^2 - 3 \eta_4 ^4}}. 
 \end{equation}
 The periodic solution $\psi$ has a fundamental period
 \begin{equation}
 	\label{2.22}
 	T_{\psi} = 2g \sqrt{c} (2 {\rm K} (\kappa)) = \frac{8 \sqrt{c} \, {\rm K} (\kappa)}{\sqrt{\eta_4(\eta_3 - \eta_1)}} =
 	\frac{8 \sqrt{c} \, {\rm K} (\kappa)}{(16c^2 \eta_4 ^2 - 3 \eta_4 ^4)^{1/4}} .
 \end{equation}
 It is follows in a straightforward manner from  (\ref{2.18}) that if $\eta_4 \to \frac{2 c}{\sqrt{3}}$, then $\eta_3 \to \frac{2 c}{\sqrt{3}}$, and hence,
 $\kappa = 0$ and $T_{\psi_c} \to \frac{2 \pi}{\sqrt{c}}$. On the other hand, if $\eta_4 \to 2 c$, then $\eta_3 \to 0$, $\kappa = 1$ and
 $T_{\psi_c} \to \infty$ (the separatrix solution).
 
 Since $\eta \in (\frac{2c}{\sqrt{3}}, 2 c) \to T_{\psi} (\eta)$
 is an increasing function (see the theorem below), it follows that
 $$
 T_{\psi} > \frac{2 \pi}{\sqrt{c}}.
 $$
 Next step is to build a smooth curve $c \to \psi$ of periodic solutions of (\ref{2.8}) with a minimal period.
 
 Let $c > 0$ be such that $\sqrt{c} > \frac{2 \pi}{L}$. Another result of the fact that
 the mapping $\eta \in (\frac{2c}{\sqrt{3}}, 2 c) \to T_{\psi} (\eta)$ is strictly
 increasing is the existence of a unique $\eta \in (\frac{2c}{\sqrt{3}}, 2 c)$ such that $T_{\psi} (\eta) = L$.
 Therefore, from (\ref{2.20}) and
 (\ref{2.21}) we get an explicit solution for (\ref{2.8}) depending only on $\eta = \eta_4$.
 
 The choice $c \to \eta (c)$ is smooth.
 \begin{thm}
 	\label{thm2.1}
 	Let $L > 0$ be arbitrary, but fixed. Consider $c_0 > \frac{4 \pi^2}{L^2}$ and unique $\eta_{4,0} = \eta (c_0) \in (\frac{2c_0}{\sqrt{3}}, 2 c_0)$
 	such that $T_{\psi_{c_0}} (\eta_{4,0} ) = L$. Then
 	
 	(1) There exists an interval $V_{c_0}$ around $c_0$, an interval $W_{\eta_{4,0}}$ around $\eta_{4,0}$ and a unique smooth function
 	$\Gamma : V_{c_0} \to W_{\eta_{4,0}}$, such that $\Gamma (c_0) = \eta_{4,0}$ and
 	\begin{equation}
 		\label{2.23}
 		\frac{8 \sqrt{c} \, {\rm K} (\kappa)}{(16c^2 \eta_4 ^2 - 3 \eta_4 ^4)^{1/4}} = L
 	\end{equation}
 	with $\eta_4 = \eta_4 (c) = \Gamma (c)$ and $\kappa^2 = \kappa^2 (c) \in (0, 1)$ given by (\ref{2.21});
 	
 	(2) The periodic solution $\psi (., \eta_4)$ determined by $\eta_4$ has a fundamental period $L$ and satisfies (\ref{2.8}). Moreover,
 	the mapping $c \in  V_{c_0} \to \psi $ is a smooth function;

 	(3) $V_{c_0}$ can be extended to $(\frac{4\pi^2}{L^2}, + \infty)$.
 \end{thm}
 {\bf Proof.} We borrow the idea for the proof from \cite{PavaNatali}(see also \cite{Natali1}).  Define an open set
 $$
 \mathcal{V} : = \left\{(\eta, c) \in \mathbb{R}^2, \quad c > \frac{4 \pi^2}{L^2}, \quad  \eta \in \left(\frac{2c}{\sqrt{3}}, 2 c \right) \right\}
 $$
 and the smooth real function $G: \mathcal{V} \to \mathbb{R}$ given by
 $$
 G (\eta, c) = \frac{8 \sqrt{c} \, {\rm K} (\kappa)}{(16c^2 \eta_4 ^2 - 3 \eta_4 ^4)^{1/4}},
 $$
 where $\kappa \in (0, 1)$ is defined in (\ref{2.21}). Denote for short $r := 16c^2 \eta_4 ^2 - 3 \eta_4 ^4$.
 Firstly, we obtain that
 \begin{equation}
 	\label{2.24}
 	\frac{d \kappa}{ d \eta} (\eta, c) = \frac{32 c^4 \eta}{\kappa \sqrt{r^3}} > 0
 \end{equation}
 Direct calculations give
 \begin{equation}
 	\label{2.25}
 	G_{\eta} (\eta, c) = \frac{8 \sqrt{c} \eta}{\sqrt[4]{r^7}} \left\{ 32 c^4 \frac{1}{\kappa} \frac{d {\rm K}}{ d \kappa} - \sqrt{r} {\rm K} (8 c^2 - 3 \eta^2)\right\}.
 \end{equation}
 For $\eta \in \left(\frac{2c}{\sqrt{3}}, 2 c \right)$ we consider two possibilities: 1) $\eta \in \left[ \frac{2\sqrt{2}c}{\sqrt{3}}, 2 c \right) $ and
 2) $\eta \in \left(\frac{2c}{\sqrt{3}},  \frac{2\sqrt{2}c}{\sqrt{3}} \right)$.
 
 In the first possibility $8 c^2 - 3 \eta^2 \leq 0$, and hence, $G_{\eta} > 0$.

 In the second possibility, the inequality holds
 $$
 (8 c^2 - 3 \eta^2) \sqrt{r} < 32 c^4,
 $$
 from which follows the estimate
 $$
 \frac{\sqrt[4]{r^7}}{8 \eta \sqrt{c}} G_{\eta} (\eta, c) =
 \left\{ 32 c^4 \frac{1}{\kappa} \frac{d {\rm K}}{ d \kappa} - \sqrt{r} {\rm K} (8 c^2 - 3 \eta^2)\right\} >
 32 c^4 \left( \frac{1}{\kappa} \frac{d {\rm K}}{ d \kappa} - {\rm K} \right) .
 $$
 Since $\frac{1}{\kappa} \frac{d {\rm K}}{ d \kappa} - {\rm K} > 0$, we get $G_{\eta} > 0$.
 
 Hence, due to Implicit Function Theorem there exists a smooth function $\Gamma$ in $V_{c_{0}}$, such that
 $G (\Gamma (c), c) = L$ for all $c \in V_{c_{0}}$. As $c_0$ is arbitrary in $(\frac{4\pi^2}{L^2}, + \infty)$, $\Gamma$ can be extended to the interval
 $(\frac{4\pi^2}{L^2}, + \infty)$.

 \begin{cor}
 	\label{cor2.1}
 	The mapping $\Gamma : V_{c_0} \to W_{\eta_{4,0}}$ defined above is a strictly increasing function.
 \end{cor}
 {\bf Proof.} From $G (\Gamma (c), c) = L$, fulfilled for all $c \in V_{c_0}$ we get
 $$
 \frac{d \Gamma (c)}{ d c} = - \frac{\partial_c G}{\partial_{\eta} G}.
 $$
 We already know that $G_{\eta} > 0$. First, we find that
 $$
 \frac{d \kappa (\eta, c)}{ d c} = - \frac{24 c \eta^2}{\kappa \sqrt{r^3}} \left(\eta^2 - \frac{4 c^2}{3} \right) < 0.
 $$
 Next, we obtain
 $$
 \partial_c G(\eta, c) = 4 \frac{2 \,c \,r \frac{d {\rm K} (\kappa)}{d \kappa} \frac{d \kappa}{d c} - 3 \eta^4 {\rm K} (\kappa)}{\sqrt{c} \sqrt[4]{r^5}} < 0,
 $$
 from where follows the claim $\frac{d \Gamma (c)}{ d c} > 0$.

 \begin{cor}
 	\label{cor2.2}
 	For $c \in (\frac{4 \pi^2}{L^2}, +\infty)$ and $\eta_4 = \eta_4 (c)$ we have that the modulus function
 	\begin{equation}
 		\label{2.26}
 		\kappa (c) = \sqrt{\frac{3 \eta_4 ^2 -8 c^2 + \sqrt{16 c^2 \eta_4 ^2 - 3 \eta_4 ^4}}{2 \sqrt{16 c^2 \eta_4 ^2 - 3 \eta_4 ^4} } }
 	\end{equation}
 	is an increasing function with respect to $c$, that is, $d \kappa/ d c > 0$.
 \end{cor}
 Next, we present explicit formulas relating  $\eta_1, \eta_3, \eta_4$ and $c$ in terms of the modulus $\kappa \in (0, 1)$ and the period $L$.
 
 From (\ref{2.21}) and (\ref{2.23}) we obtain 
 $
 c = \frac{16 {\rm K}^2 (\kappa) \sqrt{\kappa^4 - \kappa^2 + 1}}{L^2}.
 $
  Denoting
 $h (\kappa) := 4 \sqrt{\kappa^4 - \kappa^2 + 1}$ this gives
 $$
 	c = \frac{4 {\rm K}^2 (\kappa)}{L^2} h (\kappa) > 0, \qquad \frac{d c}{d \kappa} > 0,
$$
which is \eqref{2.27}. 
 Next, again from (\ref{2.21}) we get
 $$
 	 \eta_4 = \frac{8 \sqrt{2} {\rm K}^2 (\kappa)}{\sqrt{3} L^2} \sqrt{h (\kappa)[h (\kappa) + 2 (2\kappa^2 -1)]}, 
 $$
 which is \eqref{2.28}. 
 From (\ref{2.15}) and \eqref{2.20}, we have
$$
 	\beta^2 = 2 \kappa^2 \frac{\sqrt{h (\kappa) + 2(2 \kappa^2 -1)}}{\sqrt{h (\kappa) + 2(2 \kappa^2 -1)} + \sqrt{3} \sqrt{h (\kappa) - 2(2 \kappa^2 -1)}},
$$
 which is of course \eqref{2.31}.

 \section{Well-posedness of the DSW system} 
 \label{sec:3} 
We start with some preliminaries. First, it is clear that for well-posedness purposes, we can rescale the interval $(0,L)$ to be of  length $2\pi$. So, we assume that in this section $L=2\pi$ and we henceforth denote ${\mathbb T}$ to be torus $[0,2\pi]$, equipped with the normalized Lebesgue measure. 
\subsection{Preliminaries}
 Given $f\in L^2(\mathbb{T})$, we denote $f_n$ to be its Fourier coefficient.  Given $u \in C^0_t L^2_x \cap L^2_{t,x}$, we denote the space-time Fourier transform of $u$ by $\wt{u}(\tau, k)$.  We denote the Fourier inverse by $\cf^{-1}_{\tau,k}\left[\wt{u}(\tau,k)\right]$.
 
 Given two positive quantities $A, B$, $A\lesssim B$ means that $A\leq C B$ for some $C>0$.  The relation $\gtrsim$ is analogously defined.  If $A\lesssim B$ and $A \gtrsim B$, we write $A\sim B$.   $A\ll B$ (respectivly $A \gg B$) is defined similar to the relation $\lesssim$ (resp, $\gtrsim$), except that the implicit constant $C$ is assumed to be much larger than 1.  For any real number $k$, $\lan{k} := (1+k^2)^{\f{1}{2}}$.  Given any real numbers $s, b$, fractional differential operators $J_x^s$ or $J_t^b$ are defined as 
 $$
 J_x^s f = \cf^{-1}_k \left[ \lan{k}^{s} f_k\right], \ \ \ 
 J_t^b u = \cf^{-1}_{\tau,k}\left[ \lan{\tau}^{s} \wt{u}(\tau,k)\right].
 $$
 For any real numbers $s,b$, we define the KdV Bourgain space $X^{s,b}$ via the norm:
 $$
 \n{u}{X^{s,b}} = \n{\lan{k}^s \lan{\tau-k^3}^{b} \wt{u}(\tau,k)}{L^2_\tau \ell^2_k}
 $$
 For KdV Bourgain spaces,  the following embeddings are known:
 $$
 X^{0,\f{1}{3}} \hookrightarrow L^4_{t,x}, \quad X^{0+, \f{1}{2}+} \hookrightarrow L^6_{t,x}, \quad X^{0,\f{1}{2}+} \hookrightarrow L^\infty_t L^2_x
 $$
 as well as the trivial identity, $X^{0,0} = L^2_{t,x}$. 
 
 We now define functional spaces for the solution $(u,v)$.  The space for $u$ is represented by $X^{s,\f{1}{2}} \cap H^s_x L^1_{\tau}$ and the space for $v$ is represented by $H_t^{\f{1}{2}}L^2_x \cap L^2_x L^1_\tau$.   The space for $(u,v)$ is 
 $$
 Y:= \left(X^{s,\f{1}{2}} \cap H^s_x L^1_{\tau}\right) \times \left(H_t^{\f{1}{2}}L^2_x \cap L^2_x L^1_\tau\right).
 $$
  The norm of $Y$ is defined by
 $$
 \n{(u,v)}{Y} = \n{ u}{X^{s,\f{1}{2}}} + \n{\lan{k}^s \wt{u}}{\ell^2_k L^1_\tau} +\n{ J_t^{\f{1}{2}} v}{L^2_{t,x}} + \n{\wt{v}}{\ell^2_k L^1_\tau}. 
 $$
 The spaces $X^{s,\f{1}{2}}$ and $H^{\f{1}{2}}_t L^2_x$ barely fails to embed within the target space, which is $C^0_t H^s_x$ and $C^0_t L^2_x$.  In such settings, it is typical to intersect these spaces with $H^s_x L^1_\tau$ since the completion of $H^s_x L^1_\tau$ from dense smooth functions yields the embedding $H^s_x L^1_\tau \hookrightarrow C^0_t H^s_x$.  Hence,
 $$
 (u,v) \in  Y  \hookrightarrow C^0_t H^s_x\times C^0_t L^2_x
 $$
 \subsection{Reduction to the case of mean-zero initial data} 
 We precondition \eqref{eq:1} so that we can assume that $v$ satisfies the mean-zero condition for the purpose of establishing the well-posedness result.  We can achieve this by changing variable $v\to v - \int_{\bt} v\, dx$ in the above equation.   Note $\int_{} v\, dx = \int_{\bt} g\, dx = g_0$ by mean-conservation for $v$.  Then, we have
 $$
 \left| \begin{array}{l} u_t + (uv)_x  - g_0 u_x  + u_{xxx} = 0,\\
 	v_t + uu_x  = 0\\ 
 	u|_{t=0} = f \in H^s(\mathbb{T}), \quad v|_{t=0} = g- g_0 \in L^2 (\mathbb{T}). \end{array}\right.   
 $$
 
 Then, we change variable $u(t,x)\mapsto u (t,x + g_0 t)$ and note
 $$
 \p_t u (t,x + g_0 t) = u_t  (t,x + g_0 t) + g_0 u_x (t, x+ g_0t)
 $$
 Substituting this expression to above, we can remove the extra term $-g_0 u_x$ from the equation, recovering the original form of \eqref{eq:1} with an additional assumption that $\int_\bt v(t,x)\,dx =0$ for all $t\geq 0$.  The proof of Theorem~\ref{th:wp} will proceed with this assumption.  The well-posedness of original system prior to conditioning can be easily recovered because these transformations are invertible.\\
 
 We can formulate \eqref{eq:1} as:
 \begin{equation}\label{eq:2}
 	\left| \begin{array}{l} u = e^{-t\p_x^3} u_0 - \int_0^t e^{(t-s)\p_x^3} (uv)_x\, ds,\\
 		v = v_0 - \int_0^t uu_x\, ds. \end{array}\right.
 \end{equation}

 \subsection{Bilinear estimates in Bourgain spaces} 
 We begin by listing known linear estimates in Bourgain spaces:
 
\begin{lemma} \label{le:eta}For any $\eta\in \cs_t(\mathbb{R})$, $s\in \mathbb{R}$, and $T>0$, 
 	\begin{align*}
 		\n{\eta(t/T) e^{t\p_x^3} u_0 }{X^{s,b}} &\lesssim_b \n{u_0}{H^s_x}\\
 		\n{\eta(t/T) u}{X^{s,b'}} &\lesssim T^{b-b'}\n{u}{X^{s,b}}
 	\end{align*}
 	for $0<b'<b\leq \f{1}{2}$.
 \end{lemma}
 In particular, for the second statement, it will be important that we can ``spare'' some modulation weight within each estimate.  That is, whenever we use an embedding of the type $X^{0,b}\hookrightarrow X^{0,b'}$ for $0 < b' < b \leq \f{1}{2}$, we can gain a positive power of $T$.  In addition, following properties are well-known:
 	\begin{align*}
 		\n{\int_0^t e^{(t-s)\p_x^3} (uv)_x (s,\cdot)\, ds}{X^{s,\f{1}{2}} \cap H^s_x L^1_{\tau}} &\lesssim \n{(uv)_x }{X^{s,-\f{1}{2}}} + \n{\lan{k}^s\f{\cf\left[ (uv)_x\right]}{\lan{\tau - k^3}} }{\ell^2_k L^1_\tau}\\
 		\n{\int_0^t (u u _x) (s, \cdot)\, ds}{H_t^{\f{1}{2}}L^2_x \cap L^2_x L^1_\tau} &\lesssim \n{u  u_x }{H_t^{-\f{1}{2}} L^2_x} +  \n{ \f{  \cf \left[ u u_x \right]}{\lan{\tau}} }{\ell^2_k L^1_\tau}
 	\end{align*}

  To prove the local well-posedness, we need to show the following bilinear estimates:
 \begin{lemma}\label{le:1}
 	Let $0\leq s \leq \f{1}{2}$.  For $(u,v) \in X^{s, \f{1}{2}} \times H^{\f{1}{2}}_t L^2_x$, 
 	\begin{align}
 		\n{\p_x(uv)}{X^{s,-\f{1}{2}}}  + \n{\lan{k}^s \f{\cf\left[\p_x (uv)\right]}{\lan{\tau - k^3}} }{\ell^2_k L^1_\tau}&\lesssim  \n{u}{X^{s,\f{1}{2}}} \n{v}{ H^{\f{1}{2}}_t L^2_x  }\label{eq:est1}\\
 		\n{ \p_x (u^2)}{H^{-\f{1}{2}}_t L^2_x}+  \n{ \f{  \cf \left[ u u_x \right]}{\lan{\tau}} }{\ell^2_k L^1_\tau} & \lesssim \n{u}{X^{0,\f{1}{2}}}^2 \label{eq:est2}
 	\end{align}
 \end{lemma}
Furthermore, we will see that, when introducing a time-cutoff $\eta(t/T)$ within the norms of the LHS, we can gain a positive power of $T$ on the RHS by sparing some of modulational weight.

 \begin{proof}
 	We first consider \eqref{eq:est1}.  For this estimate, relevant modulational weights are $\lan{\tau - k^3}$, $\lan{\tau_1 - k_1^3}$ and $\lan{\tau_2}$ where $(\tau_1, k_1)$ is the Fourier frequency of $u$, $(\tau_2, k_2)$ is the Fourier frequency of $v$, and $(\tau,k)$ is the Fourier frequency of the product $uv$.  Hence, these frequencies must satisfy $\tau_1 + \tau_2 = \tau$, $k_1 + k_2 = k$.  Additionally, we can assume that $k_2\neq 0$ thanks to the mean-zero assumption for $v$.  To take advantage of modulational weights, we need to observe the following algebraic identity:
 	\begin{equation}\label{eq:alg1}
 	\tau - k^3 = (\tau_1- k_1^3) + \tau_2 - k_2(k_1^2+ k^2 + k_1 k),
 	\end{equation}
 	which leads to the fact:
 	$$
 	\max\{ \lan{\tau - k^3}, \lan{\tau_1 - k_1^3}, \lan{\tau_2}\} \sim |k_2|\, (k_1^2+ k^2 + k_1 k) \gtrsim \lan{k_2} \max\{k^2, k_1^2\} =: H_1.
 	$$
 	Note that $|k_2| \gtrsim \lan{k_2}$ is possible due to the assumption $k_2 \neq 0$, and $k_1^2+ k^2 + k_1 k \gtrsim  \max\{k^2, k_1^2\}$ is possible due to the algebraic identity
 	$$
 	k_1^2+ k^2 + k_1 k = \f{3}{4} k_1^2 + \left(\f{1}{2}k_1 + k\right)^2 = \f{3}{4} k^2 + \left(k_1 + \f{1}{2}k\right)^2
 	$$
Lastly, note that $H_1 \neq 0$, since this will require $k= k_1 = 0$, which forces $k_2 =0$.  Based on the relationship stated above, one of the following must be true:
\begin{enumerate}
\item $\lan{\tau - k^3} \sim H_1$
\item $\lan{\tau_1 - k_1^3} \gtrsim H_1$
\item $\lan{\tau_2}\gtrsim H_1$
\end{enumerate}

We consider these three cases separately:\\

\textbf{Case 1.} $\lan{\tau - k^3} \sim H_1$\\

Note that when restricting $\tau\in \mathbb{R}$ so that $C^{-1} H_1 \leq |\tau - k^3| \leq C H_1$ for some $C\gg 1$, $\n{ |\tau- k^3|^{-\f{1}{2}}}{L^2_{\tau}}$ is uniformly bounded and only depends on $C$, and hence $\n{\lan{\tau- k^3}^{-\f{1}{2}}}{L^2_\tau}$ is uniformly bounded.  This leads to 
 	$$
 	\n{\lan{k}^s \f{\cf_{x,t} \left[\p_x(uv)\right] }{\lan{\tau - k^3}} }{\ell^2_k L^1_\tau} \lesssim \n{\lan{k}^s\f{\cf_{x,t} \left[\p_x(uv)\right] }{\lan{\tau - k^3}^{\f{1}{2}}} }{\ell^2_k L^2_\tau} = \n{\p_{x} (uv)}{X^{s,-\f{1}{2}}} 
 	$$
This means that the second term on the LHS of \eqref{eq:est1} is controlled by the first term.  Hence, it remains to show
 $$
\n{\p_{x} (uv)}{X^{s,-\f{1}{2}}} \lesssim  \n{u}{X^{s,\f{1}{2}}} \n{v}{ H^{\f{1}{2}}_t L^2_x  }
$$
Since $\lan{\tau- k^3} \sim H_1 \geq k^2 \lan{k_2}$, for any $s\in [0, \f{1}{2}]$, 
$$
\lan{k}^s \f{ik}{\lan{\tau - k^3}^{\f{1}{2}}} \lesssim \f{\lan{k}^s}{\lan{k_1}^s \lan{k_2}^{\f{1}{2}}} \lesssim \min \{|k_1|, |k_2|\}^{-\f{1}{2}}
$$
where we have used $\max\{|k_1|, |k_2|\} \gtrsim |k|$.  This estimate yields,

$$
\n{\p_{x} (uv)}{X^{s,-\f{1}{2}}} = \n{\lan{k}^s \f{ik}{\lan{\tau - k^3}^{\f{1}{2}}} \cf_{t,x}[ uv]}{\ell^2_k L^2_\tau} \lesssim  \n{\left(J_x^{s-\f{1}{2}} u\right)v}{L^2_{t,x}} + \n{\left(J_x^s u\right)\left(J_x^{-\f{1}{2}} v\right)}{L^2_{t,x}} 
$$
Applying H\"older's inequality, followed by the Sobolev and $L^4_{t,x}$ embedding,
\begin{align*}
 \n{\left(J_x^{s-\f{1}{2}} u\right)v}{L^2_{t,x}} &\lesssim  \n{J^{s-\f{1}{2}}_x u}{L^4_{t}L^\infty_x}\n{v}{L^4_{t}L^2_x}\lesssim   \n{J^{s}_x u}{L^4_{t,x}} \n{v}{H^{\f{1}{2}}_t L^2_x}\lesssim   \n{u}{X^{s,\f{1}{3}}} \n{v}{H^{\f{1}{2}}_t L^2_x} \\
\n{\left(J_x^s u\right)\left(J_x^{-\f{1}{2}} v\right)}{L^2_{t,x}} &\lesssim \n{J^s_x u}{L^4_{t,x}}\n{J_x^{-\f{1}{2}} v}{L^4_{t,x}}\lesssim  \n{u}{X^{0,\f{1}{3}}} \n{v}{H^{\f{1}{2}}_t L^2_x} 
\end{align*}
which yields the desired RHS. 	Note that we can gain a positive power of $T$ from Lemma~\ref{le:eta}, since the final quantity on the RHS above is in terms of $\n{u}{X^{s,\f{1}{3}}}$, whereas the RHS of \eqref{eq:est1} affords us a stronger norm $\n{u}{X^{s,\f{1}{2}}}$.\\

For the remaining two cases, we know that $\lan{\tau - k^3}\not\sim H_1$.  In this case, we do not need to use the full $\f{1}{2}$-power of this modulational weight.  In fact, we only need to use $\f{1}{3}$-power of this weight in order to use the $L^4_{t,x}$ embedding.  Note
 		$$
 	\n{\lan{k}^s \f{\cf_{x,t} \left[\p_x(uv)\right] }{\lan{\tau - k^3}} }{\ell^2_k L^1_\tau} \lesssim \n{\lan{k}^s \f{\cf_{x,t} \left[\p_x(uv)\right] }{\lan{\tau - k^3}^{\f{1}{3}}} }{\ell^2_k L^2_\tau} = \n{\p_x(uv)}{X^{s,-\f{1}{3}}}. 
 	$$
Also, since $\n{\p_x(uv)}{X^{s,-\f{1}{2}}}\lesssim \n{\p_x(uv)}{X^{s,-\f{1}{3}}}$, both terms on the LHS of \eqref{eq:est1} is controlled by $\n{\p_x(uv)}{X^{s,-\f{1}{3}}}$.\\

We remark that, while $b = -\f{1}{3}$ was chosen to take advantage of the $L^4_{t,x}$ embedding, we can choose any value between $-\f{1}{2}$ and $-\f{1}{3}$ for the $X^{s,b}$ norm on the RHS above.   In fact, by selecting a slight smaller $b$, such as $-\f{2}{5}$, we can again obtain a small positive power of $T$ via Lemma~\ref{le:eta}.\\

To estimate $\n{\p_x(uv)}{X^{s,-\f{1}{3}}}$, we use duality: $\left(X^{0, -\f{1}{3}}\right)^* = X^{0,\f{1}{3}}$, which leads to
$$
\n{\p_x (uv)}{X^{s,-\f{1}{3}}} = \n{J_x^s \p_x (uv)}{X^{0,-\f{1}{3}}} = \sup_{\n{w}{X^{0,\f{1}{3}}}= 1} \left|\int_{\mathbb{T}\times \mathbb{R}} \p_x(uv) J_x^s w \, dx\,dt\right|
$$
We consider the remaining two cases for an arbitrary $w$ satisfying $\n{w}{X^{0,\f{1}{3}}}= 1$.  For simplicity of calculations, it is common to assume that all functions on the Fourier side (i.e. $\wt{u}$, $\wt{v}$, $\wt{w}$) are nonnegative.  By doing so and using Parseval's identity, we can rewrite the integral to be estimated to:
$$
\left|\int_{\mathbb{T}\times \mathbb{R}} \p_x(uv) J^s_x w \, dx\,dt\right| \leq \int_{\Gamma} |k|  \wt{u}(\tau_1, k_1) \wt{v}(\tau_2, k_2) \lan{k}^s \wt{w}(\tau,k)\, d\Gamma =: I_1
$$
where $\Gamma$ is a hyperplane given by $\{\tau_1, \tau_2,\tau \in \mathbb{R}, k_1, k_2, k\in \mathbb{Z}: \tau_1 + \tau_2 = \tau, k_1 + k_2 = k\}$ and $d\Gamma$ is the inherited measure on $\Gamma.$  We consider this estimate in the remaining two cases.\\

\noindent
 	\textbf{Case 2:}  $\lan{\tau_1 - k_1^3}\gtrsim H_1$.\\
 
In this case,
\begin{align*}
I_1 &\lesssim \int_{\Gamma} \f{|k|\lan{k}^s}{H_1^{\f{1}{2}}\lan{k_1}^s} \lan{\tau_1 - k_1^3}^{\f{1}{2}} \wt{J_x^{s} u}(\tau_1, k_1) \wt{v}(\tau_2, k_2) \wt{w}(\tau,k)\, d\Gamma\\
&\lesssim \int_{\Gamma} \min\{|k_1|, |k_2|\}^{-\f{1}{2}} \left(\lan{\tau_1 - k_1^3}^{\f{1}{2}} \wt{J_x^s u}(\tau_1, k_1)\right) \left( \wt{v}(\tau_2, k_2)\right) \wt{w}(\tau,k)\, d\Gamma\\
\end{align*}

Denoting $U:= \cf^{-1}_{k_1, \tau_1} \left[ \lan{k_1}^s \lan{\tau_1 - k_1^3} \wt{u}\right]$, the RHS above is bounded by
$$
\int_{\mathbb{T}\times \mathbb{R}} \left(J_x^{s-\f{1}{2}}  U\right) \, v\, w\, dx\,dt +\int_{\mathbb{T}\times \mathbb{R}} \left(J_x^{s}  U\right)\, \left(J_x^{-\f{1}{2}}v\right)\,  w\, dx\,dt 
$$
Applying H\"older and Sobolev inequalities, this is bounded by
\begin{align*}
&\n{J_x^{s-\f{1}{2}}  U}{L^2_t L^4_x}  \n{v}{L^4_t L^2_x} \n{w}{L^4_{t,x}} + \n{J_x^{s}  U}{L^2_{t,x}}  \n{J_x^{-\f{1}{2}} v}{L^4_{t,x}}  \n{w}{L^4_{t,x}}\\
&\lesssim \n{J_x^{s}  U}{L^2_{t,x}}  \n{v}{H^{\f{1}{2}}_x L^2_x} \n{w}{X^{0,\f{1}{3}}} + \n{J_x^{s}  U}{L^2_{t,x}}  \n{v}{H^{\f{1}{2}}_t L^2_x}  \n{w}{X^{0,\f{1}{3}}}
\end{align*}
Since $\n{J_x^{s}  U}{L^2_{t,x}} = \n{u}{X^{s,\f{1}{2}}}$, we obtain the desired inequality.\\

 	\noindent
 	\textbf{Case 3:} $\lan{\tau_2}\gtrsim H_1$.\\
 	
This estimate can be obtained from similar calculations as in Case 2.  As before, we can write
\begin{align*}
I_1 &\lesssim \int_{\Gamma}  \f{|k|\lan{k}^s}{H_1^{\f{1}{2}}\lan{k_1}^s} \wt{J_x^s u}(\tau_1, k_1) \wt{J_t^{\f{1}{2}}v}(\tau_2, k_2) \wt{w}(\tau,k)\, d\Gamma\\
&\lesssim \int_{\Gamma} \min\{|k_1|, |k_2|\}^{-\f{1}{2}} \wt{J_x^s u}(\tau_1, k_1) \wt{J_t^{\f{1}{2}}v}(\tau_2, k_2) \wt{w}(\tau,k)\, d\Gamma
\end{align*}
from which similar calculations as in Case 2 ensues.  Here, we only need to make sure that $J_t^{\f{1}{2}}v$ gets placed under $L^2_t$-norms when applying H\"older's inequality.  However, both $J_x^s u$ and $w$ can be placed under $L^4_{t}$-based norms due to the $X^{0,\f{1}{3}}\hookrightarrow L^4_{t,x}$ embedding.  Hence, all Lebesgue indices work out similarly as before.  We omit the details.  This concludes all cases to prove \eqref{eq:est1}.\\

Next, we consider  \eqref{eq:est2}.  For this estimate, we observe the algebraic identity:
 	$$
 	\tau = (\tau_1 - k_1^3) +  (\tau_2 - k_2^3) + k (k_1^2+ k_2^2 - k_1 k_2)
 	$$
Here, $(\tau_1, k_1)$ is the Fourier frequency of the first entry $u$, $(\tau_2, k_2)$ is the Fourier frequency of the second entry $u$, and $(\tau, k)$ is the frequency of the product.  Again, these variables satisfy $\tau_1 + \tau_2 = \tau$ and $k_1+k_2 = k$, as well as $k\neq 0$ from the mean-zero assumption.  From the identity above, we must have
 	$$
 	\max\{\lan{\tau}, \lan{\tau_1 - k_1^3}, \lan{\tau_2 - k_2^3}\} \gtrsim \lan{k} \max\{k_1^2, k_2^2\} =: H_2
 	$$
 	One of the three statements must be true:
\begin{enumerate}
\item $\lan{\tau} \sim H_2$
\item $ \lan{\tau_1 - k_1^3} \gtrsim H_2$
\item $ \lan{\tau_2 - k_2^3}\gtrsim H_2$
\end{enumerate}
By symmetry, we will omit the third case. 	\\

 	\noindent
 	\textbf{Case 1:}  $\lan{\tau}\sim H_2 $.\\
 	
 	Using an analogous argument as in the Case 1 proof of \eqref{eq:est1}, it suffices to bound $\n{ \p_x (u^2)}{H^{-\f{1}{2}}_t L^2_x}$.  Here, we have
$$
\n{ \p_x (u^2)}{H^{-\f{1}{2}}_t L^2_x} \lesssim \n{ \f{ik}{\lan{\tau}^{\f{1}{2}}} \cf_{t,x}[u^2]}{\ell^2_k L^2_\tau} \lesssim  \n{ \lan{k}^{-\f{1}{2}} \cf_{t,x}[u^2]}{\ell^2_k L^2_\tau} \lesssim \n{u^2}{L^2_{t,x}} \lesssim \n{u}{L^4_{t,x}}^2 \lesssim \n{u}{X^{0,\f{1}{3}}}^2
$$
where we simply disposed of the $\f{1}{2}$~derivative gain for $v$. However, we will take advantage of this additional gain in derivatives in a subsequent lemma.\\

\noindent
 	\textbf{Case 2:}  $\lan{\tau_1 - k_1^3}\gtrsim H_2$.\\

As stated under the proof of \eqref{eq:est1}, it sufficies to bound $\n{ \p_x (u^2)}{H^{-\f{1}{3}}_t L^2_x}$.  We note as before $\left| ik/H_2^{\f{1}{2}}\right| \lesssim \lan{k}^{-\f{1}{2}}$.   Denoting  $U :=\cf^{-1}_{\tau_1, k_1} \left[ \lan{\tau_1 - k_1^3}^{\f{1}{2}} \wt{u} \right]$, 
 	$$
 	\n{ \p_x (u^2)}{H^{-\f{1}{3}}_t L^2_x} \lesssim  \n{ \lan{\tau}^{-\f{1}{3}} \lan{k}^{-\f{1}{2}} \cf_{t,x}[U u] }{\ell^2_k L^2_\tau}\lesssim \n{U u}{L^{6/5}_t L^1_x}\lesssim \n{u}{X^{0,\f{1}{2}}} \n{u}{L^{3}_t L^2_x}
 	$$
where we have applied Sobolev embedding and H\"older's inequality.  Using interpolation between $X^{0,0}\hookrightarrow L^2_t L^2_x$ and $X^{0, \f{1}{2}+} \hookrightarrow L^\infty_t L^2_x$, we deduce that $X^{0,b} \hookrightarrow L^3_t L^2_x$ for some $0<b<\f{1}{2}$.  This yields \eqref{eq:est2}.
 \end{proof}

 \subsection{Proof of Theorem~\ref{th:wp}} 
 Given initial data $(u_0,v_0)$ and $0<T\ll 1$, define  $\Gamma_T(u,v) = (\Gamma_1 (u,v), \Gamma_2(u,v))$ by
 $$
 \left| \begin{array}{l} \Gamma_1 (u,v) = \eta(t/T) e^{-t\p_x^3} u_0 - \int_0^t e^{(t-s)\p_x^3} (\eta(s/T)u \eta(s/T)v)_x\, ds,\\
 	\Gamma_2(u,v) = \eta(t/T)v_0 - \int_0^t (\eta(s/T)u) (\eta(s/T)u_x) \, ds. \end{array}\right.
 $$ 

Lemmas~\ref{le:eta} and \ref{le:1} yield, $\exists \ve>0$ such that
\begin{align*}
\n{\Gamma_T(u,v)}{Y} &=  \n{\Gamma_1(u,v)}{X^{s,\f{1}{2}} \cap H^s_x L^1_\tau}+\n{\Gamma_2(u,v)}{H^{\f{1}{2}}_t L^2_x \cap L^2_x L^1_\tau} \\
&\lesssim \n{u_0}{H^s_x} + \n{v_0}{L^2_x}+ T^\ve \n{(u,v)}{Y}^2
\end{align*}

as well as 
\begin{align*}
\n{\Gamma_T(u_1,v_1) - \Gamma_T(u_2,v_2)}{Y} &\lesssim  T^\ve \n{(u_1 + u_2, v_1 + v_2) }{Y}\n{(u_1 - u_2, v_1 - v_2) }{Y}.
\end{align*}
These statements show that $\Gamma_T$ is a contraction map on $Y$ for $T$ sufficiently small.  Hence, the fixed point of $\Gamma_T$ exists and is unique in $Y$.  Continuity of solution map in $Y$-norm also follows.  This proves Theorem~\ref{th:wp}.  As a biproduct of this proof, we have also established that the local-in-time solution $(u,v)$ lies inside a ball in $Y$, where the $Y$-norm of the center of the ball is controlled by $\n{u_0}{H^s} + \n{v_0}{L^2}$.  This leads to a useful estimate:
\begin{equation}\label{eq:useful}
\n{\eta(\cdot/T) u(\cdot)}{X^{s,\f{1}{2}}\cap H^s_x L^1_\tau} + \n{\eta(\cdot/T) v(\cdot)}{H^{\f{1}{2}}_t L^2_x \cap L^2_x L^1_\tau} \lesssim \n{u_0}{H^s}+ \n{v_0}{L^2}.
\end{equation}
where $s\in [0,\f{1}{2}]$.  Here, size of $T$ only depends on $\n{u_0}{H^s}$ and $\n{v_0}{L^2}$.\\

Let us turn our attention to the special case when $(u_0, v_0) \in L^2 \times L^2$.  This case is special because $\int u^2+v^2$ is conserved, which means that  $\n{u(t)}{L^2_x}$ and $\n{v(t)}{L^2_x}$ are both bounded by some constant for $t\in [0,T]$. In this case, the size of $T$ in the local well-posedness statement of Theorem~\ref{th:wp} only depends on $\n{u_0}{L^2}$ and $\n{v_0}{L^2}$. We can first find a solution on the time interval $[0,T]$.  Then we can construct a continuation of this solution on the next time interval $[T, 2T]$ using the initial data $(u(T), v(T))$.  The next time interval can be chosen to be the same since $\n{u(T)}{L^2_x}$ and $\n{v(T)}{L^2_x}$ is bounded by the same constant. This process can be iterated indefinitely to extend the local-in-time solution in $(C^0_t L^2_x)^2$ to a global-in-time solution.  This process leads to the statement of Corollary~\ref{cor:wp}.

\subsection{Normal form transformation}
\label{sec:4}
 Heuristics of this approach relies on the fact that solution $v$ of \eqref{eq:1} portrays a substantial nonlinear smoothing effect.  That is, $v = v_0 + z$ where $z$ is much smoother than $v_0 \in L^2$.   Note that our mean-zero assumption on $v$ carries onto the new variable $z$.  We can rewrite \eqref{eq:1} as
$$ 
 \left| \begin{array}{l} u_t + u_x v + u z_x + u(v_0)_x + u_{xxx}= 0\\ z_t + u u_x = 0\\ u\vert_{t=0} = u_0, z\vert_{t=0} = 0\end{array}\right.
$$
 We will place $u_0 \in H^1$ and $v_0 \in L^2$.  The roughest nonlinear term of this system is $u(v_0)_x$, which we will remove using normal form transformation.  To this end, we define the normal form operator $T(\cdot, \cdot)$ as follows.  For any $f,g \in C^\infty(\mathbb{T})$ with mean-zero condition imposed on $g$,
 $$
 T(f,g) = \mathcal{F}^{-1}_k \left[\sum_{\Tiny\begin{array}{c}k_1 + k_2 = k\\ k_2 \neq 0\end{array}} \f{-i}{k_1^2 + k^2 + k_1 k} f_{k_1} g_{k_2} \right]
 $$
Note that the denominator of the Fourier symbol only vanishes when $k_1 = k =0$, which is only possible if $k_2 = 0$.  Hence, we avoid the singularity.  By the algebraic identity \eqref{eq:alg1}, we have
 $$
 (\p_t + \p_{xxx})T(f,g) = T((\p_t + \p_{xxx})f, g) + T(f,\p_t g) + f \p_x g. 
 $$
Substituting $u$ and $v_0$ as entries of $T(\cdot, \cdot)$,
  $$
 (\p_t + \p_{xxx}) T(u,v_0) = T(\p_x (uv), v_0) + u \p_x v_0
  $$ 
We change variable with $u = T(u,v_0) + w$, which leads to the transformed system:
 \begin{equation}\label{eq:3}
\left| \begin{array}{l} v = v_0 + z\\ u = T(u,v_0) + w\\
w_t + u_x v + u z_x  + T(\p_x (uv), v_0) + w_{xxx}= 0\\ z_t + u u_x = 0\\ w\vert_{t=0} = u_0 -T(u_0,v_0)\\ z\vert_{t=0} = 0\end{array}\right.
 \end{equation}
 
First, we need to establish mapping properties of $T(\cdot, \cdot)$.
\begin{lemma}\label{le:t}
Mapping properties of $T(\cdot,\cdot)$ are listed below:
\begin{enumerate}
\item $T: H^s_x L^1_\tau \times L^2_x \to H^{s+1}_x L^1_\tau$ is a continuous map for  $s=0, 1$.

\item $T: X^{\f{1}{2},\f{1}{2}} \times L^2_x \to X^{\f{1}{2}, \f{1}{2}}$ is a continuous map.  More specifically, for $0< T \ll 1$ and $\eta\in \mathcal{S}_t$,
\begin{align}
\n{T(\eta(\cdot/T) u, v_0)}{X^{\f{1}{2}, \f{1}{2}}} &\lesssim \n{u}{X^{0,\f{1}{2}}} \n{v}{L^2_x} + T^{\f{1}{6}-} \n{u}{X^{\f{1}{2},\f{1}{2}}} \n{v_0}{L^2_x}\label{eq:t2}
\end{align}
\end{enumerate}

\end{lemma}

\begin{proof}
First, we will show $T: H^s_x L^1_\tau \times L^2_x \to H^{s+1}_x L^1_\tau$ for which we estimate $\n{T(u,v_0)}{H^{s+1}_x L^1_\tau}$. 
$$
\cf_{x}[T(u,v_0)] (k)  =  \sum_{k_1 + k_2 = k}\f{i\lan{k}^{s+1}}{k_1^2 + k^2 + k_1 k} u_{k_1} (v_0)_{k_2} 
$$
When $s=0$, the fraction in the summand abound is bounded by $\lan{k}^{-1}$.  Applying Sobov embedding,
$$
\n{T(u,v_0)}{H^{s+1}_x L^1_\tau} \lesssim \n{ J_x^{-1} (u v_0)}{L^2_x L^1_{\tau}}  \lesssim \n{u v_0}{L^1_{\tau,x}}\lesssim  \n{u}{L^2_x L^1_\tau} \n{v_0}{L^2_x}.
$$
For $s=1$, this fraction is bounded by $1$.  Hence,
$$
\n{T(u,v_0)}{H^2_x L^1_\tau}\lesssim \n{u v_0}{L_x^2 L^1_\tau}\lesssim  \n{u}{L^\infty_x L^1_\tau} \n{v_0}{L^2_x}\lesssim \n{u}{H^1_x L^1_\tau} \n{v_0}{L^2_x}
$$
Next, we will show $T: X^{\f{1}{2},\f{1}{2}} \times L^2_x \to X^{\f{1}{2}, \f{1}{2}}$, which involves estimating $\n{T(u,v_0)}{X^{\f{1}{2},\f{1}{2}}}$ with assumptions $u \in X^{0,\f{1}{2}}$ and $v_0 \in L^2$.  Using algebraic identity \eqref{eq:alg1} with $\tau_2 = 0$ and $\tau_1 = \tau$, we have 
 	$$
 	\max\{ \lan{\tau - k^3}, \lan{\tau - k_1^3} \} \gtrsim |k_2|\, (k_1^2+ k^2 + k_1 k) \gtrsim \lan{k_2} \max\{k^2, k_1^2\} = H_1.
 	$$
We split into two cases:\\

\textbf{Case 1.}  $\lan{\tau - k^3}\sim \lan{\tau - k_1^3} \gg H_1$ \\

Denoting $U  := \cf^{-1}_{k_1,\tau} \left[\lan{\tau -k_1^3}^{\f{1}{2}} \wt{u}(\tau,k)\right]$, we have
$$
\n{T(u,v_0)}{X^{\f{1}{2},\f{1}{2}}} \lesssim \n{\f{\lan{k}^{\f{1}{2}}}{k_1^2 + k^2 + k_1 k} \cf_{t,x} \left[U  v_0\right] }{\ell^2_k L^2_{\tau}}\lesssim \n{ U v_0 }{L^2_t L^1_{x}}\lesssim \n{U}{L^2_{t,x} } \n{v_0}{L^2_{x}}.
$$

\textbf{Case 2.} $\lan{\tau -  k^3} \lesssim H_1$.  In this case, note
$$
\f{\lan{k}^\f{1}{2}\lan{\tau - k_1^3}^{\f{1}{2}}  }{ k_1^2 + k^2 + k_1 k } \lesssim \f{ \lan{k}^{\f{1}{2}} \lan{k_2}^{\f{1}{2}}}{  (k_1^2 + k^2 + k_1 k)^{\f{1}{2}}} \lesssim 1.
$$
Hence, 
$$
\n{T(u,v_0)}{X^{\f{1}{2},\f{1}{2}}} \lesssim \n{ u v_0 }{L^2_{t,x}}\lesssim \n{u}{L^2_t L^\infty_x} \n{v_0}{L^2_{x}} \lesssim \n{J_x^{\f{1}{4}+} u}{L^4_{t,x}} \n{v_0}{L^2_x}
$$
where the $u$-factor on the RHS is bounded by $\n{u}{X^{\f{1}{2}, \f{1}{3}}}$ using $L^4_{t,x}$ embedding, yielding a $T^{\f{1}{6}-}$ factor using Lemma~\ref{le:eta}.  This gives \eqref{eq:t2}.
\end{proof}

\subsection{Remaining multilinear estimates in Bourgain spaces}

Now, we turn to \eqref{eq:3}. Since $u = T(u,v_0) + w$, applying \eqref{eq:t2},
$$
\n{u}{X^{\f{1}{2}, \f{1}{2}}} \lesssim \n{u}{X^{0,\f{1}{2}}} \n{v_0}{L^2_x} + T^{\f{1}{6}-} \n{ u}{X^{\f{1}{2}, \f{1}{2}}} \n{v_0}{L^2} +\n{w}{X^{\f{1}{2},\f{1}{2}}}
$$
so that $\n{u}{X^{\f{1}{2}, \f{1}{2}}} \lesssim \n{u}{X^{0,\f{1}{2}}} \n{v_0}{L^2_x}+ \n{w}{X^{\f{1}{2},\f{1}{2}}}$ where $T>0$ only depends on $\n{v_0}{L^2_x}$.  Furthermore, 
$$
\n{u}{H^1_x L^1_\tau} \lesssim \n{T(u,v_0)}{H^1_x L^1_\tau} + \n{w}{H^1_x L^1_\tau} \lesssim \n{u}{L^2_x L^1_\tau} + \n{w}{H^1_x L^1_\tau}.
$$
Since $u \in L^2_x L^1_\tau \cap X^{0,\f{1}{2}}$ from the proof of Theorem~\ref{th:wp}, we can see that $u \in X^{\f{1}{2}, \f{1}{2}} \cap H^1_x L^1_\tau$ if $w \in X^{1, \f{1}{2}} \cap H^1_x L^1_\tau$.  This motivates  our target space for $(w,z)$ which is 
$$
(w,z) \in (X^{1, \f{1}{2}} \cap H^1_x L^1_\tau) \times  H^{\f{1}{2}}_t H^1_x
$$

The following lemma establishes that $z$ lies within the target space $H^{\f{1}{2}}_t H^1_x$.  All subsequent lemmas in this section will be able to yield a positive power of $T$ on the RHS if $\eta(t/T)$ is put within the norms on the LHS.  This implication will be apparent within the proof by the fact, for instance, that the estimate ends in the norm $X^{\f{1}{2}, \f{1}{3}}$ rather than $X^{\f{1}{2}, \f{1}{2}}$.  We will not include these within each statement.

\begin{lemma}\label{le:z}
For $u\in X^{\f{1}{2},\f{1}{2}}$, 
$$
\n{\p_x (u^2)}{H^{-\f{1}{2}}_t H^1_x} \lesssim \n{u}{X^{\f{1}{2},\f{1}{2}}}^2.
$$
\end{lemma}
\begin{proof}
We can proceed in an analogous manner to the proof of \eqref{eq:est2} where the dispersive gain factor is given by $H_2$ in the previous proof.  Consider the same two cases as before:\\

\textbf{Case 1:} $\lan{\tau} \sim H_2$.\\

In the original proof, note that we have disposed of $\lan{k}^{-\f{1}{2}}$ gain within the estimate.  Here, we will use this gain, along with the fact $|k|\lesssim \max \{|k_1|, |k_2|\}$ which immediately yields the desired estimate.\\

\textbf{Case 2:} $\lan{\tau_1 -k_1^3}\gtrsim H_2$.

In this case, we slightly modify the previous proof to accommodate for the change.  Note that we are not concerned with $L^1_\tau$ estimates, so we can keep the full $\lan{\tau}^{-\f{1}{2}}$ gain for the estimate.  We have $\lan{k}^2/ H_2^{\f{1}{2}} \lesssim \lan{k}^{\f{1}{2}}$.  Applying Sobolev embedding and denoting $U := \cf_{k_1, \tau_1}^{-1}[ \lan{\tau_1- k_1^3}^{\f{1}{2}} \wt{u}(\tau,k)]$, 
\begin{align*}
\n{\p_x (u^2)}{H^{-\f{1}{2}}_t H^1_x} &\lesssim \n{\f{\lan{k}^2}{H^{\f{1}{2}}} \cf_{x} [U u]}{L^1_t \ell^2_k} \lesssim  \n{Uu}{L^1_t H^{\f{1}{2}}_x } \\
&\lesssim  \n{J^{\f{1}{2}}_x U}{L^2_t L^2_x }\n{u}{L^2_t L^\infty_x }+\n{ U}{L^2_t L^4_x }\n{J^{\f{1}{2}}_x u}{L^2_t L^4_x }\\
&\lesssim  \n{u}{X^{\f{1}{2}, \f{1}{2}} }\n{J_x^{\f{1}{4}+} u}{L^2_t L^4_x }+\n{ J_x^{\f{1}{4}} U}{L^2_t L^2_x }\n{J^{\f{1}{2}}_x u}{X^{\f{1}{2},\f{1}{3}} } \lesssim  \n{u}{X^{\f{1}{2}, \f{1}{2}} }^2
\end{align*}

\end{proof}

It remains to obtain necessary estimate for the $w$ equation in \eqref{eq:3}.  Writing  $u = T(u,v_0) + w$ in one of the occurences of $u$, the equation for $w$ is written as 
\begin{equation}\label{eq:ww}
w_t  + w_{xxx}= - w_x v - u z_x   - (T(u,v_0))_x v - T(\p_x (uv), v_0)
\end{equation}
We note that there are two quadratic terms and two trilinear terms on the RHS above.  We handle these estimates in two separate lemmas.
\begin{lemma} 
	\label{le:7} 
Let $u\in X^{\f{1}{2},\f{1}{2}}$, $v \in H^{\f{1}{2}}_t L^2_x$, $w \in X^{1,\f{1}{2}} \cap H^1_x L^1_\tau$, and $z \in H^{\f{1}{2}}_t H^1_x$.  Also, assume that $v$ and $z$ satisfy the mean-zero condition.  Then, 
\begin{align}
  \n{w_x v }{X^{1, -\f{1}{2}}} +  \n{\lan{k}\f{\cf[w_x v]}{\lan{\tau - k^3}} }{\ell^2_k L^1_\tau }  &\lesssim \n{w}{X^{1,\f{1}{2}}} \n{v}{H^{\f{1}{2}}_t L^2_x}\label{eq:est5}\\  
  \n{u z_x }{X^{1, -\f{1}{2}}}+\n{\lan{k}\f{\cf[u z_x]}{\lan{\tau -k^3}}  }{\ell^2_k L^1_\tau} &\lesssim \n{u}{X^{0,\f{1}{2}}} \n{z}{H^{\f{1}{2}}_t H^1_x}\label{eq:est6}
  \end{align}  
\end{lemma}

We remark that these bilinear estimates are both equivalent to \eqref{eq:est1}.  For instance, if $w\in X^{1, \f{1}{2}}$, then $w_x \in X^{0,\f{1}{2}}$.  Therefore, estimating $\n{w_x v}{X^{1,-\f{1}{2}}}$ for  is equivalent to estimating $\n{\p_x (uv)}{X^{0,-\f{1}{2}}}$.  Hence, \eqref{eq:est5} follows from \eqref{eq:est1}.  The same can be said about \eqref{eq:est6} since $z\in H^{\f{1}{2}}_t H_x^1$ implies $z_x \in H^{\f{1}{2}}_t L^2_x$.\\

Next, we consider the necessary trilinear estimates.
  
\begin{lemma} 
	\label{le:8} 
Let $u\in X^{0,\f{1}{2}}$, $v_0 \in L^2_x$, and $v \in H^{\f{1}{2}}_t L^2_x$.  Also, assume that $v_0$ and $v$ satisfy the mean-zero condition. 
\begin{align}
  \n{(T(u,v_0))_x v }{X^{1, -\f{1}{2}}} + \n{\lan{k}\f{\cf[(T(u,v_0))_x v ]}{\lan{\tau - k^3}} }{\ell^2_k L^1_\tau }&\lesssim \n{u}{X^{0,\f{1}{2}}}\n{v_0}{L^2_x} \n{v}{H^{\f{1}{2}}_t L^2_x}\label{eq:est7}\\  
   \n{T(\p_x (uv), v_0) }{X^{1, -\f{1}{2}}}+ \n{\lan{k}\f{\cf[T(\p_x (uv), v_0)]}{\lan{\tau - k^3}} }{\ell^2_k L^1_\tau } &\lesssim \n{u}{X^{0,\f{1}{2}}} \n{v}{H^{\f{1}{2}}_t L^2_x} \n{v_0}{L^2_x}\label{eq:new}
  \end{align}
  
\end{lemma}

\begin{proof}
First, consider \eqref{eq:est7}.  We begin by observing the corresponding algebraic identity:
\begin{equation} \label{eq:alg2}
\tau - k^3 = (\tau_1 -k_1^3) + \tau_3 - (k_2 + k_3)(k_1^2 + k^2 + k_1 k)
\end{equation}
where $\tau_1 + \tau_3 = \tau$, $k_1 + k_2 + k_3 = k$ and $k_2, k_3 \neq 0$.  Hence, the trilinear modulation estimate is given by
$$
\max\{|\tau - k^3|,|\tau_1 -k_1|, |\tau_3|\}\gtrsim |k_2 + k_3| \max\{k_1^2, k^2\} =: H_3 
$$
We note that $H_3$ can vanish $k_2 + k_3= 0$.  Such frequency interaction is referred to as a trilinear resonance, where nonlinear dispersive effects cancel each other, leaving the solution with zero dispersion at the very specific frequency interaction regime.  While we can assume $k_2, k_3 \neq 0$ due to the mean-zero assumption, we cannot assume that $k_2 + k_3 \neq 0$.\\

Before splitting into cases as we had done previously, let us first manage the estimate in the resonant case $k_2 + k_3=0$.  Here, we do not need to use the modulational weights, and we can control both terms on the LHS of \eqref{eq:est7} by $\n{J_x \left[(T(u,v_0))_x v \right]}{L^2_{t,x}}$.  For all cases of frequency interaction, we can write
\begin{equation}\label{eq:tri1}
\cf_x [ J_x\left[ (T(u,v_0))_x v\right]] (k) =  \sum_{\tiny \begin{array}{c} k_1 + k_2 + k_3 = k\\ k_2, k_3 \neq 0\end{array}} \f{ (k_1 + k_2)\lan{k}}{k_1^2 + (k_1+ k_2)^2 + k_1 (k_1 + k_2)} u_{k_1} (v_0)_{k_2} v_{k_3}.
\end{equation}
Specifically for the resonant case, the expression reduces to 
$$
 \sum_{ k_2 \neq 0} \f{ (k + k_2)\lan{k}}{k^2 + (k+ k_2)^2 + k (k + k_2)} u_{k} (v_0)_{k_2} v_{-k_2}.
$$
Note that the fraction within the summand above is bounded by 1.  Applying Cauchy-Schwartz inequality for $k_2$, we obtain
\begin{align*}
\n{J_x (T(u,v_0))_x v }{L^2_{t,x}} &\lesssim \n{ u  \n{v_0}{L^2_x} \n{v}{L^2_x}   }{L^2_{t,x}} \lesssim \n{ u   }{L^4_t L^2_x} \n{v_0}{L^2_x} \n{v}{L^4_t L^2_x} \lesssim \n{ u   }{X^{0,\f{1}{2}}} \n{v_0}{L^2_x} \n{v}{H^{\f{1}{2}}_t L^2_x}
\end{align*}

For the remaining cases, we can assume that the resonance does not occur (that is, $k_2 + k_3 \neq 0$).   Considering \eqref{eq:tri1}, the fraction within the summand is bounded by $\lan{k}/\max\{|k_1|, |k_1+k_2|\}$.  Unlike the resonant case, this quantity may be unbounded, in particular when $|k|\sim |k_3| \gg \max\{|k_1|, |k_2|\}$.  In this case, our modulational gain (that is, $H_3$) must be used to control this weight. More precisely, we will use the following estimate:
\begin{equation}\label{eq:triest1}
\left|\f{ (k_1 + k_2)\lan{k}}{H_3^{\f{1}{2}} \left(k_1^2 + (k_1+ k_2)^2 + k_1 (k_1 + k_2)\right)}\right| \lesssim \lan{k_1}^{-1} \lan{k_2 + k_3}^{-\f{1}{2}}
\end{equation}
  We split into the following cases:
\begin{enumerate}
\item $\lan{\tau - k^3}\sim H_3$
\item $\lan{\tau_1 -k_1}\gtrsim H_3$
\item $\lan{\tau_3} \gtrsim H_3$
\end{enumerate}

\textbf{Case 1:} $\lan{\tau - k^3}\sim H_3$.\\

As before, in this case, it suffices to estimate $\n{(T(u,v_0))_x v }{X^{1, -\f{1}{2}}}$.  Using \eqref{eq:triest1},
\begin{align*}
\n{(T(u,v_0))_x v }{X^{1, -\f{1}{2}}}&\lesssim \n{(J_x^{-1}u) J_x^{-\f{1}{2}} (v_0 v) }{L^2_{t,x}}\lesssim \n{J_x^{-1}u}{L^4_{t} L^\infty_x} \n{ J_x^{-\f{1}{2}} (v_0 v) }{L^4_{t} L^2_x}\\
&\lesssim \n{u}{L^4_{t,x}} \n{ v_0 v }{L^4_{t} L^1_x} \lesssim \n{u}{L^4_{t,x}} \n{ v_0}{L^2_x} \n{ v }{L^4_{t} L^2_x}\lesssim \n{u}{X^{0,\f{1}{3}}} \n{ v_0}{L^2_x} \n{ v }{H^{\f{1}{4}}_{t} L^2_x}
\end{align*}

For the remaining two cases, we can again control the LHS of \eqref{eq:est7} by $\n{(T(u,v_0))_x v }{X^{1, -\f{1}{3}}}$ and apply duality.  The mechanism of proof for both Case 2 and 3 are analogous, so we will only show the proof of Case 3 below.

\textbf{Case 3:} $\lan{\tau_3}\gtrsim H_3$.\\

Applying \eqref{eq:triest1} and duality,
\begin{align*}
\n{(T(u,v_0))_x v }{X^{1, -\f{1}{3}}}&\lesssim \sup_{\n{\vp}{X^{0,\f{1}{3}}} = 1} \left| \int_{\mathbb{T}\times \mathbb{R}} (J_x^{-1}u) J_x^{-\f{1}{2}} (v_0 J_t^{\f{1}{2}}v) \vp\, dx\,dt \right|\\
&\lesssim \sup_{\n{\vp}{X^{0,\f{1}{3}}} = 1}  \n{J_x^{-1}u}{L^4_t L^\infty_x} \n{J_x^{-\f{1}{2}} (v_0 J_t^{\f{1}{2}}v)}{L^2_t L^2_{x}} \n{\vp}{L^4_{t} L^2_x}\\
&\lesssim \n{u}{L^4_{t,x}} \n{ v_0}{L^2_x}  \n{v}{H^{\f{1}{2}}_t L^2_{x}} 
\end{align*}
	
This proves \eqref{eq:est7}.\\

Proof of \eqref{eq:new} is very similar and even a bit simpler, and much of the detail can be omitted.  Below, we will only note the similarities between these estimates, from with the proof of \eqref{eq:new} will directly follow.\\

First, to consider the modulational weight,  the same algebraic identity as \eqref{eq:alg2} applies, except with $\tau_3$ replaced by $\tau_2$.   Hence, the trilinear modulation estimate is given by
$$
\max\{|\tau - k^3|,|\tau_1 -k_1|, |\tau_2|\}\gtrsim |k_2 + k_3| \max\{k_1^2, k^2\} = H_3 
$$
While we still have to deal with the resonance problem where $k_2 + k_3= 0$, the situation is very similar to before. The Fourier coefficient is written as:
\begin{equation}\label{eq:tri2}
\cf_x [ J_x\left[ T(\p_x(uv),v_0) v\right]] (k) =  \sum_{\tiny \begin{array}{c} k_1 + k_2 + k_3 = k\\ k_2, k_3 \neq 0\end{array}} \f{ (k_1 + k_2)\lan{k}}{(k_1+k_2)^2 + k^2 + (k_1+k_2) k } u_{k_1} v_{k_2} (v_0)_{k_3}.
\end{equation}
For the resonant case, the expression reduces to 
$$
 \sum_{ k_2 \neq 0} \f{ (k + k_2)\lan{k}}{(k+k_2)^2 + k^2 + k (k + k_2)} u_{k} v_{k_2} (v_0)_{-k_2}.
$$
Note that in both the general case \eqref{eq:tri2} and the resonant case, the fraction within the summand is uniformly bounded.  The estimate in the resonant case follows exactly the same calculations as before. For the non-resonant case, the use of $H_3$ weight yields the same estimate as \eqref{eq:triest1}, hence the remaining calculations are identical.
\end{proof}

\subsection{Proof of Theorem \ref{prop:109}}
Since we already know that $\n{v(t)}{L^2_x}$ is bounded  by $\n{u_0}{L^2}+ \n{v_0}{L^2}$ for all $t\geq 0$, we focus on obtaining the bound for $\n{u(t)}{H^1_x}$.   First, using $u = T(u,v_0) + w$ and the remarks following Lemma~\ref{le:t}, we can write 
\begin{align*}
\n{u}{C^0_t H^1_x} &\lesssim \n{T(u,v_0)}{H^1_x L^1_\tau} + \n{w}{H^1_x L^1_\tau} \lesssim \n{u}{L^2_x L^1_\tau} \n{v_0}{L^2_x}  +\n{w}{H^1_x L^1_\tau}\\
\n{u}{X^{\f{1}{2},\f{1}{2}}}&\lesssim \n{u}{X^{0,\f{1}{2}}} \n{v_0}{L^2_x} + \n{w}{X^{1,\f{1}{2}}} 
\end{align*}
Using \eqref{eq:ww} as well as Lemma~\ref{le:7} and \ref{le:8}, we have 
\begin{align*}
\n{w}{H^1_x L^1_\tau\cap X^{1,\f{1}{2}}} &\lesssim \n{u_0}{H^1} \left(1 + \n{v_0}{L^2}\right)   + T^{0+} \left(\n{u}{X^{0,\f{1}{2}}} \n{v_0}{L^2_x}\n{v}{H^{\f{1}{2}}_t L^2_x}  + \n{w}{X^{1,\f{1}{2}}} \n{v}{H_t^{\f{1}{2}} L^2_x} \right. \\
&+ \left.\n{u}{X^{0,\f{1}{2}}}\n{z}{H^{\f{1}{2}}_t H^1_x} + \n{u}{X^{0, \f{1}{2}}} \n{v}{H^{\f{1}{2}}_t L^2_x} \n{v_0}{L^2_x}\right)
\end{align*}
where the positive power $T$ is gained from including a $\eta(t/T)$ factor on the LHS.  Also, Lemma~\ref{le:z} gives
$$
\n{z}{H^{\f{1}{2}}_t H^1_x} \lesssim  T^{0+}\n{u}{X^{\f{1}{2},\f{1}{2}}}^2
$$

Applying \eqref{eq:useful} with $s=0$, we can bound $\n{u}{X^{0,\f{1}{2}}\cap L^2_x L^1_\tau} + \n{v}{H^{\f{1}{2}}_t L^2_x \cap L^2_x L^1_\tau}$ by $\n{u_0}{L^2} +\n{v_0}{L^2}$ where the choice of  $T$ depends only on $\n{u_0}{L^2} +\n{v_0}{L^2}$.  Hence,
\begin{align*}
\n{w}{H^1_x L^1_\tau\cap X^{1,\f{1}{2}}} &\lesssim \n{u_0}{H^1} \left(1 + \n{v_0}{L^2}\right)   + T^{0+} \left(\n{u_0}{L^2} \n{v_0}{L^2_x}^2 + \n{u}{X^{\f{1}{2},\f{1}{2}}}^3 \right)
\end{align*} 
Substituting this to estimate $u$, we obtain
$$
\n{u}{C^0_t H^1_x\cap X^{\f{1}{2},\f{1}{2}}} \lesssim \left(\n{u_0}{L^2} \n{v_0}{L^2}  +  \n{u_0}{H^1}\right) \left(1 + \n{v_0}{L^2}\right)  + T^{0+} \n{u}{X^{\f{1}{2}, \f{1}{2}}}^3
$$
Finally, we apply \eqref{eq:useful} with $s= \f{1}{2}$, where the choice of $T$ depends on $\n{u_0}{H^{\f{1}{2}}} +\n{v_0}{L^2}$.  We can use this estimate to bound $\n{u}{X^{\f{1}{2}, \f{1}{2}}}^2$ and choose $T$ small, based on $\n{u_0}{H^{\f{1}{2}}}$ to move this term to the LHS of the inequality.  This achieves the bound:
\begin{equation}\label{eq:h1bound}
\n{u(t)}{C^{0}_t H^1_x} \lesssim C\left( \n{u_0}{H^1}, \n{v_0}{L^2}\right) \qquad \text{ for } 0\leq t\leq T.
\end{equation}
Note that the size of $T$ depends on $\n{u_0}{H^{\f{1}{2}}}$, and bound on the RHS depends on $\n{u_0}{H^1}$.  In order to extend this result for all $t\geq 0$, we need to control $\n{u_0}{H^{1}}$.  The argument for the a priori bound on $\n{u(t)}{H^1_x}$ is outlined as follows: \\

 Due to the conservation law 
   \begin{eqnarray*}
   & & 	   \int  u_x^2(t) – u(t) v(t)^2\, dx =  \int  (u_0')^2 – u_0 v_0^2\, dx=:M_0 \\
   	& & \int u^2+v^2 = \int u_0^2+v_0^2=:M_1.
   \end{eqnarray*}
 As a result, 
   \begin{eqnarray*}
   	  \|u_x(t, \cdot)\|^2 &\leq &  M_0+ \|u(t, \cdot)\|_{L^\infty} \int v^2(t) dx= \\
   	  &=& M_0+ \|u(t, \cdot)\|_{L^\infty} \int v_0^2 dx\leq M_0+ 2  \|v_0\|^2 \|u_x(t, \cdot)\|\|u(t)\|=
   	  M_0+ 2  M_1^{\f{3}{2}}  \|u_x(t, \cdot)\|. 
   \end{eqnarray*}
Clearly, one can then hide $\|u_x(t, \cdot)\|$ on the right-hand side behind $\|u_x(t, \cdot)\|^2$, to obtain 
 $$
 \|u_x(t, \cdot)\|^2\leq 10 (M_0+8 M_1^3 )=:M_2^2
 $$
This tells us that $\n{u(t)}{H^1_x}$ is bounded for all $t\geq 0$.  By iterating this scheme for fixed-size time intervals, we can extend \eqref{eq:h1bound} for all $t\geq 0$.  This proves Theorem~\ref{prop:109}.


 \section{Spectral stability - preparations} 
 \label{sec:5}
 We begin with a discussion of the  properties of the operator $\cl_+$.
 \subsection{Spectral properties of $\cl_+$}

 \begin{proposition}
 	\label{p:10} 
 	The operator $\cl_+$ has one dimensional kernel, and moreover $Ker(\cl_+)=span[\psi']$. 
 \end{proposition}
As we need to develop further tools, we  postpone the proof of Proposition \ref{p:10} for the Appendix. 
\begin{proposition}
	\label{p:9} 
	For all values of the parameters, as in Proposition \ref{prop:10}, the operator $\cl_+$ has exactly two negative e-values.
\end{proposition}
For the somewhat technical proof of Proposition \ref{p:9}, we refer to the Appendix. 
\subsection{Spectral properties of $\ch$}
The first result is a direct consequence of Proposition \ref{p:10}. 
 \begin{proposition}
 	\label{prop5}
 	Let $\psi$ be a wave satisfying \eqref{2.5}. Then $dim(\ker(\ch))=1$ and in fact, 
 	$$
 	\ker(\ch)=span\begin{pmatrix} \psi' \\ \phi' \end{pmatrix}.
 	$$
 \end{proposition}
 \begin{proof}
 Let   $ \left( \begin{array}{cc}f \\ g \end{array}\right)\in \ker \mathcal{H}$, i.e. 
 	\begin{eqnarray*}
 	\cl_- f - \psi g = 0 \\
 		-\psi f + c g=0
 	\end{eqnarray*}
 	From the second equation, we have $g = \frac{1}{c}\psi f$ and plugging in the second equation, we get
 	\begin{equation}
 		\label{hill1}
 		\left( \cl_- -\frac{1}{c}\psi^2\right)f=\cl_+ f = 0.
 	\end{equation}
 Thus, by Proposition \ref{p:10}, we know $f=c_0 \psi'$ and $g=c_0 \f{\psi \psi'}{c}$, whence 
 $$
 Ker[\ch]=span\begin{pmatrix} \psi' \\ \frac{1}{c}\psi\psi' \end{pmatrix}=
 span\begin{pmatrix} \psi' \\ \phi' \end{pmatrix}
 $$
 \end{proof}
Our second result establishes the Morse index of $\ch$, namely the number of negative eigenvalues of $\ch$, $n(\ch)=\# \{\la\in \si(\ch): \la<0\}$.
\begin{proposition}
	\label{p:11} 
	For all values of the parameters, as in Proposition \ref{prop:10}, the operator $\ch$ has exactly two negative e-values. That is,  $n(\ch)=2$. 
\end{proposition}
\begin{proof}
	The starting point is Proposition \ref{p:9}. Specifically,  we show first that since $n(\cl_+)=2$, then $n(\ch)\leq 2$. Denote the lowest two eigenvalues of $\cl_+$ and their corresponding eigenfunctions by $-\si_j^2, \chi_j, j=1,2$. That is $-\si_0^2<-\si_1^2<0$ and 
	$$
	\cl_+ \chi_j=-\si_j^2 \chi_j, j=0,1
	$$
	Clearly, for all $f\perp \chi_j, j-1,2$, we have $\dpr{\cl_+ f}{f}\geq 0$. 
	
	Consider then the quadratic form associated
	to the $\mathcal{H}$ 
	\begin{equation}
		\label{m1}
		\begin{array}{ll}
			\langle \mathcal{H} \left( \begin{array}{cc} f \\ g \end{array} \right), \left( \begin{array}{cc} f \\ g \end{array} \right) \rangle &=\langle \cl_- f,f \rangle+c\langle g,g\rangle-2\langle \psi f, g\rangle \\
			\\
			&=\langle \mathcal{L}_+ f, f\rangle+\int_{-\frac{L}{2}}^{\frac{L}{2}}{\left( \sqrt{c}g-\frac{1}{\sqrt{c}}\psi f\right)^2}dx\geq \langle \mathcal{L}_+ f, f\rangle
		\end{array}
	\end{equation}
Thus, consider  $z=\begin{pmatrix}
	z_1 \\ z_2
\end{pmatrix}\in span[\begin{pmatrix}
	\chi_0 \\ 0
\end{pmatrix}, \begin{pmatrix}
	\chi_1 \\ 0
\end{pmatrix}]^\perp$, i.e. $\dpr{z_1}{\chi_j}=0, j=0,1$. We deduce that 
$$
\dpr{\ch z}{z}\geq \dpr{\cl_+ z_1}{z_1}\geq 0
$$
It follows by the min-max principles that $n(\ch)\leq 2$. 

It now remains to show that conversely, $\ch$ has at least two negative eigenvalues. To this end, let us consider the e-value problem associated with $\ch$. For $\mu>0$, and let $-\mu\in \si(\ch)$, i.e. 
	\begin{eqnarray*}
	\cl_- f - \psi g = -\mu f \\
	-\psi f + c g=-\mu g
\end{eqnarray*}
Solving for $g$, we obtain the following scalar equation for $f$, 
\begin{equation}
	\label{p:30} 
\cl_\mu f:=	\cl_+ f+\mu f + \f{\mu}{c(c+\mu)} \psi^2 f=0.
\end{equation}
This, we have that $-\mu \in \si(\ch)$ if and only if the problem \eqref{p:30} has a solution $f\neq 0$, or equivalently, the Hill's operator 
$$
\cl_\mu:= \cl_+  +\mu  + \f{\mu}{c(c+\mu)} \psi^2 
$$
has eigenvalue zero. Note that the function $\mu\to L_\mu$ is increasing, i.e. if $\mu_1>\mu_2$, then $\cl_{\mu_1}>\cl_{\mu_2}$ in operator sense. For any $\mu\geq 0$, denote the lowest two eigenvalues of $\cl_\mu$ as follows
$$
\la_0(\mu):=\la_0(\cl_\mu)<\la_1(\mu):=\la_1(\cl_\mu).
$$
Both of these functions are continuous $\la_j: [0, +\infty)\to \rone, j=0,1$. Also, note $\la_0(0)=\la_0(\cl_0)=-\si_0^2$, while for $\la=\si_0^2$, we have 
$$
\cl_\mu> \cl_+ +\si_0^2 \geq 0,
$$
whence $\la_0(\si_0^2)>0$. Thus, we have a change of sign in $(0, \si_0^2)$ for the continuous function $\la_0(\cdot)$, whence there is $\mu_0\in (0, \si_0^2)$, so that $\la_0(\mu_0)=0$. It follows that $-\mu_0\in \si(\ch)$, a negative e-value. Furthermore, consider $\la_1:[0, \mu_0]\to \rone$. Clearly, 
$$
\la_1(0)=\la_1(\cl_+)=-\si_1^2<0, \la_1(\mu_0)>\la_0(\mu_0)=0,
$$
whence change of sign occurs for $\la_1$ in the interval $[0, \mu_0]$. Thus, there is $\mu_1\in (0, \mu_0)$, so that $\la_1(\mu_1)=0$, whence $-\mu_1\in \si(\ch)$ as well. Note that $-\mu_0<-\mu_1<0$, and so we have produced two negative e-values for $\ch$. The proof is complete.

\end{proof}

Our next result identifies by a direct inspection,  the generalized kernel of the operator $\p_x \ch$. To this end, observe that by Proposition \ref{prop5}, we clearly have\footnote{Here $\ch^{-1} \begin{pmatrix} 1\\ 0  \end{pmatrix},  \ch^{-1} \begin{pmatrix} 0\\ 1  \end{pmatrix}$  are uniquely defined in the co-dimension one subspace $Ker(\ch)^\perp=\begin{pmatrix} \psi'\\ \phi'  \end{pmatrix}^\perp$}
$$
Ker(\p_x \ch)=  span[\ch^{-1} \begin{pmatrix} 1\\ 0  \end{pmatrix},  \ch^{-1} \begin{pmatrix} 0\\ 1  \end{pmatrix}, \begin{pmatrix} \psi'\\ \phi'  \end{pmatrix}].
$$
\begin{proposition}
	\label{prop:21}
	The vectors 
	$$
e_1:= \begin{pmatrix} \cl_+^{-1} [1] \\ c^{-1} \psi  \cl_+^{-1} [1] \end{pmatrix}, \ \ \ e_2= \begin{pmatrix}  c^{-1} \cl_+^{-1} \psi \\ c^{-1}+ c^{-2} \psi \cl_+^{-1} \psi \end{pmatrix}, e_3= 
\begin{pmatrix}   
	\f{1}{2c^2} \cl_+^{-1}\psi^3+   \cl_+^{-1}\psi \\ 
 \f{1}{2c^3} \psi \cl_+^{-1}\psi^3+\f{1}{c} \psi \cl_+^{-1}\psi+\f{\psi}{2c^2}\end{pmatrix}, 
	$$
are solutions of 
	$$
	\ch e_1 = \begin{pmatrix} 1\\ 0  \end{pmatrix},  \ch e_2 = \begin{pmatrix} 0\\ 1  \end{pmatrix},  \ch e_3 = \begin{pmatrix} \psi\\ \phi  \end{pmatrix}
	$$
	Furthermore, 
	\begin{eqnarray}
		\label{l:1}
		\dpr{\ch e_1}{e_1} &=& \dpr{\cl_+^{-1} 1}{1} \\
		\label{l:2}
		\dpr{\ch e_1}{e_2} &=& c^{-1} \dpr{\cl_+^{-1} 1}{\psi} \\
		\label{l:3}
			\dpr{\ch e_1}{e_3} &=&     \dpr{\cl_+^{-1} 1}{\psi}  +\f{1}{2c^2} \dpr{\cl_+^{-1}\psi^3}{1} \\
			\label{l:4}
		\dpr{\ch e_2}{e_2} &=& L c^{-1} + c^{-2} \dpr{\cl_+^{-1} \psi}{\psi}, \\
		\label{l:5}
			\dpr{\ch e_2}{e_3} &=&   \f{1}{2c^3} \dpr{ \cl_+^{-1}\psi^3}{\psi} +\f{1}{c}\dpr{ \cl_+^{-1}\psi}{\psi} +\f{\int \psi}{2c^2} \\
			\label{l:6}
			\dpr{\ch e_3}{e_3} &=&  \f{1}{c^2}\dpr{\cl_+^{-1}\psi^3}{\psi}+\dpr{ \cl_+^{-1}\psi}{\psi}+\f{1}{4 c^4}\dpr{\cl_+^{-1}\psi^3}{\psi^3}+\f{\int\psi^3}{4c^3}.
	\end{eqnarray}
\end{proposition}
{\bf Remark:} An equivalent way of stating Proposition \ref{prop:21} is to say that the first generation vectors in the generalized kernel (but not in $Ker(\ch)$) are exactly $e_j, j=1,2,3$. Moreover, under certain non-degeneracy condition, there are no more  generalized e-vectors. More precisely, we have the following Proposition.
\begin{proposition}
	\label{prop:83} 
	Suppose that the symmetric matrix $\cd=\{\dpr{\ch e_i}{e_j}\}_{i,j=1,2,3}$ is non-degenerate. Then
	\begin{equation}
		\label{p:20} 
		gKer(\p_x \ch)\ominus Ker(\ch)=span[e_1, e_2, e_3].
	\end{equation}
\end{proposition}
\begin{proof}
	In Proposition \ref{prop:21} we have verified the first generation generalized eigenvectors to be $e_1, e_2, e_3$. The task now is to show that these are in fact all generalized eigenvector.Assuming the opposite, for a contradiction, we have that for some non-trivial triple  $(\zeta_1, \zeta_2, \zeta_3)\neq (0,0,0)$, so that 
	\begin{equation}
		\label{m:10} 
		\p_x \ch \vec{z}=\sum_{j=1}^3 \zeta_j e_j.
	\end{equation}
Since 
$$
\dpr{\p_x \ch \vec{z}}{\begin{pmatrix} 1\\ 0  \end{pmatrix}}=0, \ \ \dpr{\p_x \ch \vec{z}}{\begin{pmatrix} 0\\ 1  \end{pmatrix}}=0, \ \ \dpr{\p_x \ch \vec{z}}{\begin{pmatrix} \psi\\ \phi  \end{pmatrix}}=0, 
$$
we conclude that 
$$
\sum_{j=1}^3 \zeta_j e_j\perp span[\begin{pmatrix} 1\\ 0  \end{pmatrix}, 
\begin{pmatrix} 0\\ 1  \end{pmatrix}, \begin{pmatrix} \psi\\ \phi  \end{pmatrix}]=span[\ch e_1, \ch e_2, \ch e_3]. 
$$
Writing out the inner products to be zero, yields the system 
$
\sum_{j=1}^3  \dpr{\ch e_k}{ e_j} \zeta_j=0, k=1,2,3, 
$
or equivalently 
$$
\cd\vec{\zeta}=\cd\begin{pmatrix}
	\zeta_1 \\ \zeta_2 \\\zeta_3
\end{pmatrix}=0.
$$
But by assumption, $\cd$ is non-singular, hence $\vec{\zeta}=0$, a contradiction. Hence, \eqref{p:20} holds. 
\end{proof}

\section{Spectral stability - computation of the instability index}
\label{sec:6}
We start with a discussion of the general properties of the eigenvalue problem \eqref{41}. 
\subsection{Introduction to the instability index count}
With respect to the standard inner product on $L^2[0,L]$, we have  $(\p_x)^*=-\p_x, \ch=\ch^*$, whence \eqref{41} is a Hamiltonian eiganvalue problem. In particular, it obeys the symmetry, that if $\la\in \si(\p_x \ch)$, then $-\la, \pm \bar{\la} \in \si(\p_x \ch)$. Regarding the number of unstable eigenalues, under the assumption that a lot of work has been done in the last twenty years or so, \cite{Pel, Kap}, culminating with \cite{LZ}. In order to state the relevant result, introduce the following spectral counts 
\begin{eqnarray*}
	k_r &=& \# \{\la\in \si_{p}(\p_x \ch): \la>0\} \\
	k_c &=& \# \{\la\in \si_p(\p_x \ch): \Re \la>0, \Im \la>0\}\\
	k_{i}^- &=& \# \{i \mu, \mu>0: \cj\cl f=i \mu f, \dpr{\cl f}{f}<0 \}.
\end{eqnarray*}
That said, and assuming that the genealized eigenspace 
$$
gKer(\p_x \ch)=\cup_{l=1}^\infty \{u: (\p_x \ch)^l u=0, l=1, 2, \ldots\}
$$
is finite dimensional, introduce a basis $gKer(\p_x \ch)=span\{\eta_j: j=1, \ldots, N\}$. It turns out that a symmetric matrix $\cd$, defined via 
$$
\cd_{i j}:=\dpr{\ch \eta_i}{\eta_j}, i,j=1, \ldots, N
$$
encodes an important relevant information, namely the formula 
\begin{equation}
	\label{krein} 
	k_{Ham.}:=k_r+2k_c+2k_i^-=n(\ch)-n(\cd),
\end{equation}
where $n(\ch), n(\cd)$ are the Morse indices\footnote{Assumed to be finite} of the operator $\ch$ and the 
$N\times N$ matrix $\cd$ respectively. 
Note that by Proposition \ref{p:11}, we have that $n(\ch)=2$, whence 
\begin{equation}
	\label{l:18} 
	k_{Ham.}=2-n(\cd).
\end{equation}
In addition,  note that for the purposes of computing one may remove vectors $\eta_k\in Ker(\ch)$ as the quantity $n(\cd)$ is unaffected by the zero rows and columns corresponding to $\eta_k\in Ker(\ch)$. So, formula \eqref{p:20} provides three such generalized eigenfunctions, namely $e_1, e_2, e_3$, which is a complete system, provided some non-degeneracy conditions are met (to be verified later). 

That means, again by Proposition \ref{prop:21},  that it is enough to consider the matrix $\cd$ with entries given by \eqref{l:1} through \eqref{l:6}. Our plan is to compute these quantities in terms of the parameters $\ka\in (0,1)$ and $L\in (0, \infty)$. This is possible, due to the formulas \eqref{2.27}, \eqref{2.28}, \eqref{2.31}, \eqref{2.a10}. 

Our next task is to find a way to compute the inner products appearing in the formulas for $\dpr{\ch e_i}{e_j}$, \eqref{l:1},..., \eqref{l:6}, which involve the operator $\cl_+^{-1}$. Thus, we aim at constructing a computable framework for this operator. Recall that by the translational invariance\footnote{or better yet by taking derivative in $x$ in the profule equation \eqref{2.5}}, we find that $\cl_+[\psi']=0$. The task is then to find another, linearly independent element in the $Ker(\cl_+)$. As expected, such element {\it is not a peridic function}, but it will allow us, modulo some minor modifications to construct $\cl_+^{-1}$. Details are below in Lemma \ref{le:po}.  

\subsection{Construction of $\cl_+^{-1}$}
\begin{lemma}
	\label{le:po} 
	The function $\vp:[0,L]\to \rone$, 	
	\begin{eqnarray*}
& & 	 \varphi(x) =  \frac{1}{2\alpha^2 \eta_4(\kappa^2+\beta^2)} \left[\frac{[1+\beta^2sn^2(\alpha x)]^2[1-2sn^2(\alpha x)]}{dn^2(\alpha x)} \right.\\
	& - & \left. \alpha \frac{sn(\alpha x)cn(\alpha x)dn(\alpha x)}{[1+\beta^2sn^2(\alpha x)]^2}  \int_{0}^{x}{\frac{[1+\beta^2sn^2(\alpha s)]^3[3(\kappa^2+\beta^2)+5\beta^2dn^2(\alpha s)][1-2sn^2(\alpha s)}{dn^4(\alpha s)}}ds \right].
	\end{eqnarray*} 
where $\al=\f{2K(\ka)}{L}$, satisfies 
$
\cl_+[\vp]=0.
$

 As a consequence, we can define the inverse $\cl_+^{-1}$, {\bf for an even function $f$}, as follows
  \begin{eqnarray}
  	\label{l:90}
  	 \mathcal{L}_+ ^{-1} f (x) &=&  \psi'(x) \int_{0}^{x}{\varphi(s)f(s)}ds-\varphi(x)\int_{0}^{x}{\psi'(s)f(s)}ds+
  	 C_f\varphi(x), \\
  	 \nonumber
  	 C_f &=& \int_{0}^{\frac{L}{2}}\psi'f-\frac{\psi''(\frac{L}{2})}{2\varphi'(\frac{L}{2})}\langle \varphi ,f\rangle.
  \end{eqnarray}
\end{lemma}
{\bf Remarks:} 
\begin{itemize}
	\item The choice of $C_f$ is made so that the function on the right-hand side of \eqref{l:90} is $L$ periodic, and hence belongs to  $D(\cl_+)=H^2_{per.}[0,L]$. 
	\item The function $\vp$, while smooth in $[0,L)$ is notably {\bf not an $L$ periodic function}, as alluded to above. In fact, while it is the case that  $\vp(0)=\vp(L)$, we shall be able to  check that $\vp'(0)\neq \vp'(L)$. In fact,   in the proof below, we show that it suffices to verify the technical condition \eqref{l:955} to conclude that $\vp'(0)\neq \vp'(L)$. 
\end{itemize}
\begin{proof}
	Since $ \mathcal{L}_+ \psi' = 0 $,  we have by D'Alembert's formula the function $\varphi$
	\begin{equation}
		\label{l:92} 
			\varphi(x) := \psi'(x)\int^{x}{\frac{1}{\psi'^2(s)}}ds, \; \; \left| \begin{array}{cc} \psi'& \varphi \\ \psi'' & \varphi'\end{array}\right|=1
	\end{equation}
{\it formally} satisfies $\cl_+ \vp=0$. 
	Formula, such as \eqref{l:92} is problematic, because  $\psi'$ has  zeros. Instead,   by using the identities
	$$
	\frac{1}{sn^2(y,\kappa)}=-\frac{1}{dn(y, \kappa)}\frac{\partial}{\partial y}\frac{cn(y, \kappa)}{sn(y, \kappa)}, \; \; \frac{1}{cn^2(y,\kappa)}=\frac{1}{dn(y, \kappa)}\frac{\partial}{\partial y}\frac{sn(y, \kappa)}{cn(y, \kappa)}$$
	and after  integrating by parts, we get an equivalent to \eqref{l:92}  formula 
	$$
	\begin{array}{ll} \varphi(x) &=\frac{1}{2\alpha^2 \eta_4(\kappa^2+\beta^2)}\left[\frac{[1+\beta^2sn^2(\alpha x)]^2[1-2sn^2(\alpha x)]}{dn^2(\alpha x)}\right. \\
		\\
		& \left.- \alpha \frac{sn(\alpha x)cn(\alpha x)dn(\alpha x)}{[1+\beta^2sn^2(\alpha x)]^2}  \int_{0}^{x}{\frac{[1+\beta^2sn^2(\alpha x)]^3[3(\kappa^2+\beta^2)+5\beta^2dn^2(\alpha x)][1-2sn^2(\alpha s)]}{dn^4(\alpha s)}}ds\right].
	\end{array}
	$$
	which introduces an even function $\vp$. Clearly $\vp(0)=\vp(L)$, while looking at $\vp'$ through the product rule, we see that it contains $L$ periodic functions, with the exception of 
	$$
	\f{d}{dx} \left[\frac{sn(\alpha x)cn(\alpha x)dn(\alpha x)}{[1+\beta^2sn^2(\alpha x)]^2} \right] \int_{0}^{x}{\frac{[1+\beta^2sn^2(\alpha x)]^3[3(\kappa^2+\beta^2)+5\beta^2dn^2(\alpha x)][1-2sn^2(\alpha s)]}{dn^4(\alpha s)}}ds.
	$$
	Since $\f{d}{dx} \left[\frac{sn(\alpha x)cn(\alpha x)dn(\alpha x)}{[1+\beta^2sn^2(\alpha x)]^2} \right]|_{x=L}\neq 0$, it remains to observe  that under the condition, 
	\begin{equation}
		\label{l:955} 
		\int_{0}^{2K(\ka)}{\frac{[1+
				\beta^2 sn^2(s)]^3[3(\kappa^2+\beta^2)+5\beta^2
				dn^2(s)][1-2sn^2(s)]}{dn^4(s)}}ds\neq 0, 
	\end{equation}
	we have that indeed $\vp'(0)\neq \vp'(L)$, and so $\vp$ is not $L$ periodic. 
\end{proof}
Note that we need to compute $\cl_+^{-1}1, \cl_+^{-1}\psi$, all even functions, so  formula \eqref{l:90} applies.  Taking inner product with $1$ and $\psi$ respectively and integration by parts implies 
 \begin{equation}
	\label{es1}
	\left\{ \begin{array}{ll}
		\langle \mathcal{L}_+^{-1} 1,1\rangle=-2\langle \psi, \varphi\rangle+\left( 2\psi (\frac{L}{2})-\frac{\psi''(\frac{L}{2})}{2\varphi'(\frac{L}{2})}\langle \varphi,1\rangle\right)\langle \varphi,1\rangle \\
		\\
		\langle \mathcal{L}_+^{-1}\psi,1\rangle=-\frac{3}{2}\langle\psi^2,\varphi\rangle+\frac{1}{2}\psi^2(\frac{L}{2})\langle\varphi,1\rangle+\left( \psi(\frac{L}{2})-\frac{\psi''(\frac{L}{2})}{2\varphi'(\frac{L}{2})}\langle \varphi,1\rangle\right)\langle \psi, \varphi\rangle \\
		\\
		\langle \mathcal{L}_+^{-1}\psi,\psi\rangle=-\langle\psi^3,\varphi\rangle+\left( \psi^2(\frac{L}{2})-\frac{\psi''(\frac{L}{2})}{2\varphi'(\frac{L}{2})}\langle \psi, \varphi\rangle\right)\langle \psi, \varphi\rangle
	\end{array} \right.
\end{equation}
Let us record the formula $\psi^3=-c\mathcal{L}_+\psi-cF_1$. Applying $\cl_+^{-1}$ to it leads to 
\begin{equation}
	\label{l:95} 
	\cl_+^{-1} [\psi^3]=-c \psi-cF_1 \cl_+^{-1} [1].
\end{equation}
Taking inner product with $1, \psi$  leads to the expressions 
\begin{eqnarray}
	\label{l:98}
	\langle \mathcal{L}_+^{-1}\psi^3,1\rangle=-c\langle \psi,1\rangle-c F_1\langle \mathcal{L}_+^{-1} 1,1\rangle \\
	\label{l:99}
	\langle \mathcal{L}_+^{-1}\psi^3,\psi\rangle=-c\langle \psi,\psi\rangle-c F_1\langle \mathcal{L}_+^{-1}\psi, 1\rangle 
\end{eqnarray}
which reduces to the expressions in \eqref{es1}. 
Taking inner product with $\psi^3$ yields 
\begin{equation}
	\label{l:101} 
	\dpr{\cl_+^{-1} \psi^3}{\psi^3}=-c\int \psi^4-c F_1 	\dpr{\cl_+^{-1} \psi^3}{1},
\end{equation}
which itself reduces to \eqref{l:98}. 

Thus, based on the formulas \eqref{es1}, \eqref{l:98},  \eqref{l:99},  \eqref{l:101}, we can in compute all the   entries in the matrix $\cd$, provided 
\subsection{Computing the matrix $\cd$ and the Hamiltonian index}
Our first task is to compute \\ $\dpr{\vp}{1}, \dpr{\vp}{\psi}$.
We have found no easier way other than to use the explicit formula provided in Lemma \ref{le:po}, which yields after integration by parts to 
\begin{eqnarray}
	\label{l:102}
	\langle \varphi, 1\rangle &=& \frac{L^3}{8K^3(\kappa)\eta_4 (\kappa^2+\beta^2)}[A_1+
	\frac{1}{2(\kappa^2+\beta^2)}\frac{1-\kappa^2}{1+\beta^2}A_2-\frac{1}{2(\kappa^2+\beta^2)}A_3]\\
	\label{l:104} 
	\langle \varphi, \psi\rangle &=& \frac{L^3}{8K^3(\kappa)(\kappa^2+\beta^2)}[A_4+
	\frac{1}{4(\kappa^2+\beta^2)}\frac{(1-\kappa^2)^2}{(1+\beta^2)^2}A_5-\frac{1}{4(\kappa^2+\beta^2)}A_6],
\end{eqnarray}
where 
\begin{eqnarray*}
	A_1 &=& \int_{0}^{K(\kappa)}{\frac{[1+\beta^2sn^2x]^2[1-2sn^2x]}{dn^2x}}dx\\
	A_2 &=& \int_{0}^{K(\kappa)}{\frac{[1+\beta^2sn^2x]^3[3(\kappa^2+\beta^2)+5\beta^2dn^2x][1-2sn^2x]}{dn^4x}}dx\\
	A_3 &=& \int_{0}^{K(\kappa)}{\frac{[1+\beta^2sn^2x]^2[3(\kappa^2+\beta^2)+5\beta^2dn^2x][1-2sn^2x]}{dn^2x}}dx\\
	A_4 &=& \int_{0}^{K(\kappa)}{[1+\beta^2sn^2x]^2[1-2sn^2x]}dx\\
	A_5 &=& \int_{0}^{K(\kappa)}{\frac{[1+\beta^2sn^2x]^3[3(\kappa^2+\beta^2)+5\beta^2dn^2x][1-2sn^2x]}{dn^4x}}dx\\
	A_6&=& \int_{0}^{K(\kappa)}{[1+\beta^2sn^2x][3(\kappa^2+\beta^2)+5\beta^2dn^2x][1-2sn^2x]}dx.
\end{eqnarray*}
Next, note the integration by parts formula for an even function $h$,
 implies\footnote{Note the extra boundary term, due to the fact that $\vp$ is not periodic}
\begin{equation}
	\label{l:100} 
	\langle\varphi'', h\rangle=2 h (\frac{L}{2})\varphi'(\frac{L}{2})+\langle\varphi,h''\rangle.
\end{equation}
Now, taking inner product of \eqref{l:95} with $\vp$, and  the integration by parts formula \eqref{l:100}, taking into account $\cl_+[\vp]=0$ yields 
\begin{equation}
	\label{l:105}
	\langle \psi^3,\varphi\rangle=- 2c\psi \left(\frac{L}{2}\right) \varphi'\left(\frac{L}{2}\right) - c F_1\langle\varphi,1\rangle.
\end{equation}
Similarly, $\dpr{\cl_+ \vp}{1}=0$, integration by parts \eqref{l:100} and $\cl_+ 1=c - \f{3 \psi^2}{2c}$, imply 
\begin{equation}
	\label{l:106}
\langle \psi^2, \varphi \rangle = -\frac{4c}{3} \varphi' \left(\frac{L}{2}\right) + \frac{2c^2}{3} \langle \varphi, 1 \rangle.
\end{equation}
Note that the remaining quantities involved in the formulas are directly calculated from the respective definitions as follows 
  \begin{eqnarray}
  	\label{l:110} 
  	\psi(\frac{L}{2}) &=& \eta_4\frac{1-\kappa^2}{1+\beta^2}, \ \ \ 
  	\varphi'(\frac{L}{2})=\frac{L (1-\ka^2)}{4\eta_4 K(\kappa) (\ka^2+\be^2)(1+\be^2)^2} A_2\\
  \label{l:112}
  	\psi''(\frac{L}{2}) &=& \frac{8\eta_4K^2(\kappa)}{L^2}\frac{(1-\kappa^2)(1+\beta^2)+\beta^2(1-\kappa^2)^2}{(1+\beta^2)^2}\\
  \label{l:114} 
  	\langle\psi,1\rangle &=& \frac{\eta_4L}{K(\kappa)}\int_{0}^{K(\kappa)}\frac{dn^2x}{1+\beta^2sn^2x}dx, \ \ \ \  \	\langle \psi, \psi\rangle =  \frac{\eta_4^2L}{K(\kappa)}\int_{0}^{K(\kappa)}\frac{dn^4x}{[1+\beta^2sn^2x]^2}dx, \\
  	\label{l:116} 
  \int \psi^3 &=& \frac{\eta_4^3 L}{K(\kappa)}\int_{0}^{K(\kappa)}\frac{dn^6 x}{[1+\beta^2sn^2x]^3}dx, \ \ 
  \int \psi^4 =\frac{\eta_4^4 L}{K(\kappa)}\int_{0}^{K(\kappa)}\frac{dn^8 x}{[1+\beta^2sn^2x]^4}dx 
  \end{eqnarray}
 Finally, the formulas for $\eta_4, \be^2, F_1, c$, in terms of $\ka, L$,  can be found in the statement of Proposition \ref{prop:10}, namely 
 \eqref{2.27}, \eqref{2.28}, \eqref{2.31}, \eqref{2.a10}. 
 
 This allows us, in principle and after putting it all together the entries of the matrix $\cd$. Note that the period $L$ enters as a multiplicative constant everywhere, and can be pulled out of a calculation involving the eigenvalues of $\cd$. For the remaining paramter, we use  the symbolica and numerical capabilities of  Mathematica to compute,  as the formulas and the calculations are too long and convolved.  
 
 To illustrate the instability, we found convenient to compute the determinant of $\cd$.  We verify, through the numerics outlined above  - formulas \eqref{l:1}, \ldots, \eqref{l:6}, followed by \eqref{es1}, \ldots,  \eqref{l:105} and finally \eqref{l:110}, \dots  \eqref{l:116}. The result, and the corresponding relevant graphs are displayed below, for $L=1$. 
 It is clear that as $\det(\cd(\ka))<0$, the matrix $\cd$ has exactly one negative eigenvalue, so $n(\cd)=1$, whence by formula \eqref{l:18}, we find that $k_{Ham.}=1$, thus ensuring spectral instability for the eigenvalue problem \eqref{41}. 
	\begin{figure}
	\centering
	\includegraphics[width=0.7\linewidth]{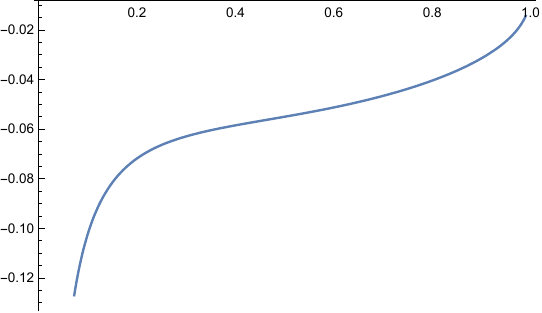}
	\caption{Graph of 	$\ka\to \det(\cd(\ka))$ }
	\label{fig:a2}
\end{figure}

\appendix
\section{Proof of Proposition \ref{p:10}}
We have already verified, by direct inspection that $\psi'\in Ker(\cl_+)$. It remains to show that there is no other element in $Ker(\cl_+)$. As explained in the remarks after Lemma \ref{le:po},   it suffices to show that $\vp$ is not a periodic function, which amounts to \eqref{l:955}. Since the integrand is even and $2K(\ka)$ periodic, matters reduce to checking 
$$
A_2(\ka) =	\int_{0}^{K(\ka)}{\frac{[1+
		\beta^2 sn^2(s)]^3[3(\kappa^2+\beta^2)+5\beta^2
		dn^2(s)][1-2sn^2(s)]}{dn^4(s)}}ds\neq 0.
$$
We have, in the course of our computations for the matrix $\cd$,  have evaluated $A_2$ as a function of $\ka$. As one can see, $A_2(\ka)<0$ for all values of the parameter,  see Figure \ref{fig:a1}. 
	\begin{figure}
	\centering
	\includegraphics[width=0.7\linewidth]{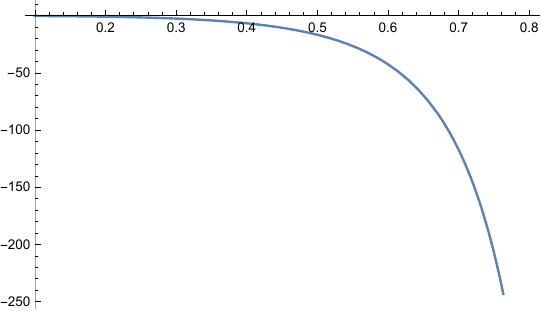}
	\caption{Graph of 	$A_2(\ka)$ }
	\label{fig:a1}
\end{figure}

\section{Proof of Proposition \ref{p:9}}
We briefly review  some known definitions, facts and
results about Hill operators.
The Hill operator
\begin{equation}
	\label{3.5}
	\mathcal{L}_c (y) = - y'' + Q (c, x) y
\end{equation}
with a smooth $L$-periodic with respect to $x$ potential $Q (c,
x)$ has a spectrum
$$
\lambda_0 < \lambda_1 \leq \lambda_2 < \lambda_3 \leq \lambda_4
\ldots < \lambda_{2n-1} \leq \lambda_{2n} < \ldots
$$
and equality means that $\lambda_{2n-1} = \lambda_{2n}$ is a
double eigenvalue \cite{Magnus}. Suppose $p (x)$ is an
eigenfunction of $\mathcal{L}_c$ associated to the eigenvalue
$\lambda$. If $\lambda = \lambda_{2n-1}$ or $\lambda =
\lambda_{2n}$ then
$p (x)$ has exactly $2n$ zeroes in $[0, L)$.
It is important to know whether any properties of $p$  help to
establish whether $\lambda$ is simple or not, and if it is simple,
whether $\lambda = \lambda_{2n-1}$ or $\lambda = \lambda_{2n}$. A
suitable tool to answer these questions is developed by Neves
\cite{Neves2}. We will come back to this result in just a moment.

Let $n_{-}$ be the dimension of the negative subspace of
$\mathcal{L}_c$ and $n_0$ be the dimension of the null subspace of
$\mathcal{L}_c$. The inertial index of $\mathcal{L}_c$ is defined
as $in (\mathcal{L}_c) := (n_{-}, n_0)$. A family of self-adjoint
operators $\mathcal{L}_c$ depending on a parameter $c$ in certain
interval $V \subset \mathbb{R}$, is called {\it isoinertial} if
the the inertial index does not depend on the parameter. The
following theorem gives a sufficient condition for isoinertiality.
\begin{thm}
	\label{thm3.1} (Neves \cite{Neves1}) Let $\mathcal{L}_c$ be the
	Hill operator (\ref{3.5}) defined in $L^2 _{per} ([0, L])$ with
	domain $D (\mathcal{L}_c) = H^2 _{per} ([0, L])$. If $\lambda =0$
	is an eigenvalue of $\mathcal{L}_c$ for every $c$ in an open
	interval $V \subset \mathbb{R}$ and the potential $Q (c, x)$ is
	continuously differentiable with respect to all variables, then
	the family of operators $\mathcal{L}_c, c \in V$ is isoinertial.
\end{thm}
In our case the operator $\mathcal{L}_c$ takes the form
\begin{equation}
	\label{3.15}
	\cl_+=\mathcal{L}_c: = - \partial_{\xi \xi} + c -
	\frac{3}{2 c} \psi ^2 (\xi).
\end{equation}
Recall that $\lambda = 0$ is an eigenvalue of $\mathcal{L}_c$ with
the eigenfunction $p (\xi) = \psi' (\xi)$. 
Direct computations from (\ref{2.16}) give that
$$
\psi ' (\xi) = - \frac{\eta_4 (\kappa^2 + \beta^2 )}{a \sqrt{c}}
\frac{sn (\frac{\xi}{2a\sqrt{c}}) cn (\frac{\xi}{2ga\sqrt{c}}) dn
	(\frac{\xi}{2a\sqrt{c}})} {\left(1 + \beta^2 sn^2
	(\frac{\xi}{2a\sqrt{c}})\right)^2} .
$$
Since $\psi ' (\xi)$ satisfies a linear equation, we can drop the
constant in the above expression for simplicity, so
\begin{equation}
	\label{3.19}
	p (\xi) = \frac{sn (\frac{\xi}{2a \sqrt{c}}) cn
		(\frac{\xi}{2a\sqrt{c}}) dn (\frac{\xi}{2a\sqrt{c}})} {\left(1 +
		\beta^2 sn^2 (\frac{\xi}{2a\sqrt{c}})\right)^2} .
\end{equation}
Also, 
\begin{equation}
	\label{3.20}
	p (0) = 0, \qquad p' (0) =  \frac{1}{2 a \sqrt{c}}.
\end{equation}
Clearly, $p (\xi)$ has two zeroes in $[0, L)$. Then the Floquet
Theory suggests that there are three possibilities:
\begin{enumerate}
	\item  $\lambda_0 < \lambda_1 = \lambda_2 = 0 \rightarrow
	in(\mathcal{L}_c) = (1, 2)$;
	\item  $\lambda_0 < \lambda_1 = 0 < \lambda_2  \rightarrow
	in(\mathcal{L}_c) = (1, 1)$;
	\item 
	$\lambda_0 < \lambda_1  < \lambda_2 = 0 \rightarrow
	in(\mathcal{L}_c) = (2, 1)$;
\end{enumerate}
We are interested in the number of negative eigenvalues
related to $\mathcal{L}_c$.
To take advantage of Theorem \ref{thm3.1} we consider the Hill
equation
\begin{equation}
	\label{3.21}
	- y'' (\xi) + \left(c - \frac{3}{2 c} \psi ^2 (\xi)
	\right) y (\xi) = 0.
\end{equation}
We already know that $p (\xi)$ is a L-periodic solution of
(\ref{3.21}) for every $c$. The classical Floquet Theory
\cite{Magnus} states that, there is another smooth solution $q
(\xi)$, such that $q (\xi), p (\xi)$ are with unit Wronskian $Wr
(q (\xi), p (\xi)) = 1$ and related with
\begin{equation}
	\label{3.22} q (\xi + L) = q (\xi) + \theta p (\xi)
\end{equation}
fulfilled for every $\xi \in \mathbb{R}$. Evidently, $q (\xi)$ is
L-periodic if, and only if, $\theta = 0$. It is worth mentioning
that formulas for $q (\xi)$ and $\theta$ are given in
\cite{Neves2}. We slightly reformulate the result of Neves
(\cite{Neves2}, Theorem 3.1) to fit our purposes.
\begin{thm}
	\label{thm3.2} (Neves \cite{Neves2}) If $p (\xi)$ is the
	eigenfunction of $\mathcal{L}_c$ associated to the eigenvalue
	$\lambda = 0$ and $\theta$ is the constant from  (\ref{3.22}).
	Then $\lambda = 0$ is simple if and only if $\theta \neq 0$.
	Moreover, if $p (\xi)$ has $2n$ zeroes in the half-open interval
	$[0, L)$, then $\lambda = 0 = \lambda_{2n-1}$ if $\theta < 0$ and
	$\lambda = 0 = \lambda_{2n}$ if $\theta > 0$.
\end{thm}
In our case $p (\xi)$ has two  zeroes, then the operator
$\mathcal{L}_c$ would have only one negative eigenvalue if we show
that $\theta < 0$ for some fixed $c$. The first eigenvalue is
always simple, by the above result the eigenvalue $\lambda = 0$ is
also simple, and therefore since $\mathcal{L}_c$ is isoinertial
from Theorem \ref{thm3.1}, then $in (\mathcal{L}_c) = (1, 1)$ for
all $c > 4 \pi^2 /L^2$.
On the other hand, if we show that $\theta > 0$ for some fixed $c$, by above argument
$in (\mathcal{L}_c) = (2, 1)$ for all $c > 4 \pi^2 /L^2$.

To determine the sign of $\theta$ we proceed as follows. First, we
have to find the unique solution $q (\xi)$ (linearly independent
of $p (\xi)$) of the initial value problem
\begin{eqnarray}
	\label{3.23}
	& - q '' & + \left(c - \frac{3}{2 c} \psi ^2 (\xi)\right) q = 0, \nonumber \\
	& q (0) & = \frac{1}{p' (0)}, \\
	& q' (0) & = 0 \nonumber
\end{eqnarray}
Then,  after differentiating (\ref{3.22}) and putting $\xi =0$ we
get for $\theta$
\begin{equation}
	\label{3.24}
	\theta = \frac{q' (L)}{p' (0)},
\end{equation}
where $L > 0$ is arbitrary, but fixed period of $\psi$.

We fix $L > 0$ and $\kappa \in (0, 1)$. In the end of previous
section we have established a bijective correspondence between
$\kappa$ and $c$. Through (\ref{2.27}) we find the value of $c$,
corresponding to these $L$ and $\kappa$. 
Then the IVP (\ref{3.23}) is solved numerically with Mathematica
package to obtain $q' (L)$.

In Table \ref{table:1} we collect some values of the constant
$\theta$ for different choices of $L$ and $\kappa$. In all choices
$\theta$ turns out to be positive, which implies that $\mathcal{L}_c$ has two
negative eigenvalues and the eigenvalue $\lambda = 0$ is
simple for the corresponding values of $c$.

\begin{table}[h!]
	\centering
	\begin{tabular}{ | c| c | c | c | c | c|}
		\hline
		L & $\kappa$ & c & $p' (0)$  & $q' (L)$ & $\theta$ \\
		\hline
		2   & 0.1 & 9.87007 & 1.57475 & 0.00205921 & 0.00130764 \\
		2   & 0.2 & 9.87731 & 1.58687 & 0.03585840 & 0.022597   \\
		2   & 0.3 & 9.91068 & 1.60805 & 0.210449   & 0.130873   \\
		3   & 0.1 & 4.3867  & 1.04983 & 0.00205916 & 0.00196142 \\
		3   & 0.2 & 4.38992 & 1.05791 & 0.0358583  & 0.0338954 \\
		3   & 0.3 & 4.40475 & 1.07203 & 0.210449   & 0.196309 \\
		4   & 0.1 & 2.46752 & 0.787373& 0.00205919 & 0.00261527 \\
		4   & 0.2 & 2.46933 & 0.793434 & 0.0358584 & 0.0451939 \\
		4   & 0.3 & 2.47767 & 0.804024 & 0.210449  & 0.261745 \\
		4   & 0.5 & 2.56152 & 0.842875 & 2.77357   & 3.29061 \\
		4   & 0.7 & 2.95039 & 0.92284  & 30.2488   & 32.7777 \\
		10  & 0.1 & 0.394803& 0.314949 & 0.0205909 & 0.00653786 \\
		10  & 0.2 & 0.395092& 0.317374 & 0.0358584 & 0.112985 \\
		10  & 0.4 & 0.400374& 0.328    & 0.830212  & 2.25113 \\
		50  & 0.1 & 0.0158037& 0.0634747 & 0.0358582 & 0.564921\\
		\hline
	\end{tabular}
	
	\vspace{0.3cm}
	
	\caption{Values of $\theta$ related to the period $L$ and the
		modulus $\kappa$.} \label{table:1}
\end{table}
To summarize, we have calculated for particular values of $c$ that
the constant $\theta$ is positive. By Theorem \ref{thm3.2} we
conclude that $\mathcal{L}_c$ has two negative
eigenvalues and $\lambda = 0$ is also simple eigenvalue. Since
$\mathcal{L}_c$ is isoinertial by Theorem \ref{thm3.1} we deduce
that $in (\mathcal{L}_c) = (2, 1)$ for all $c > 4 \pi^2 /L^2$. The
reminder of the spectrum is discrete and bounded away from zero.

\end{document}